\documentclass[12pt,reqno]{article}

\usepackage[usenames]{color}
\usepackage{amssymb}
\usepackage{amsmath}
\usepackage{amsthm}
\usepackage{amsfonts}
\usepackage{amscd}
\usepackage{graphicx}

\usepackage[colorlinks=true,
linkcolor=webgreen,
filecolor=webbrown,
citecolor=webgreen]{hyperref}

\definecolor{webgreen}{rgb}{0,.5,0}
\definecolor{webbrown}{rgb}{.6,0,0}

\usepackage{color}
\usepackage{fullpage}
\usepackage{float}

\usepackage{graphics}
\usepackage{latexsym}

\begin{document}

\theoremstyle{plain}
\newtheorem{theorem}{Theorem}
\newtheorem{corollary}[theorem]{Corollary}
\newtheorem{lemma}{Lemma}
\newtheorem{proposition}{Proposition} 
\newtheorem{example}{Examples}
\newtheorem*{remark}{Remark}

\begin{center}
\vskip 1cm{\LARGE\bf 
Binomial sum relations involving Fibonacci and Lucas numbers \\
}
\vskip 1cm
{\large
Kunle Adegoke \\
Department of Physics and Engineering Physics, \\ Obafemi Awolowo University, Ile-Ife, Nigeria \\
\href{mailto:adegoke00@gmail.com}{\tt adegoke00@gmail.com}

\vskip 0.25 in
 
Robert Frontczak \\
Independent Researcher \\
Reutlingen,  Germany \\
\href{mailto:robert.frontczak@web.de}{\tt robert.frontczak@web.de}

\vskip 0.25 in

Taras Goy  \\
Faculty of Mathematics and Computer Science\\
Vasyl Stefanyk Precarpathian National University, Ivano-Frankivsk, Ukraine\\
\href{mailto:taras.goy@pnu.edu.ua}{\tt taras.goy@pnu.edu.ua}}
\end{center}

\vskip .2 in

\begin{abstract} 
In this paper, we introduce relations between binomial sums involving (generalized) Fibonacci and Lucas numbers, 
and different kinds of binomial coefficients. We also present some relations between sums with two and three binomial coefficients.
In the course of exploration we rediscover a few relations presented as problem proposals.
\end{abstract}

\vskip 0.2cm

	\section{Introduction and motivation}
	
	The literature on Fibonacci numbers is immensely rich. There exist dozens of articles and problem proposals dealing with binomial sums involving these sequences as (weighted) summands. We attempt to give a short survey, not claiming completeness. 
	The following binomial sums have been studied ($X_n$ stands for a (weighted) Fibonacci or Lucas number, alternating or non-alternating, or a product of them):
	\begin{itemize}
		\item Standard form and variants of it \cite{Adegoke1,Adegoke2,Carlitz,CarFer,Hoggatt,Kilic2,Layman,Long}
	\begin{equation*}
			\sum_{k=0}^n \binom{n}{k} X_{k};
			\end{equation*}
		\item Forms coming from the Waring formula and studied by Gould \cite{Gould1}, for instance,
			\begin{equation*}
			\sum_{k=0}^{n-1}\frac{n}{n-k} \binom{n-k}{k} X_{k};
			\end{equation*}
		\item Forms introduced by Filipponi \cite{Fil}
			\begin{equation*}
			\sum_{k=0}^n \binom{2n-k-1}{k}X_{k};
			\end{equation*}
		\item Forms introduced by Jennings \cite{Jen}
			\begin{equation*}
			\sum_{k=0}^n \binom{n+k}{2k} X_{k};
			\end{equation*}
		\item Forms introduced by Kilic and Ionascu \cite{Kilic1}
			\begin{equation*}
			\sum_{k=0}^n \binom{2n}{n+k} X_{k};
			\end{equation*}
		\item Forms studied recently by Bai, Chu and Guo \cite{Bai}
			\begin{equation*}
			\sum_{k=0}^{\lfloor n/2 \rfloor} \binom{2n}{n-2k} X_{k} \quad \mbox{and}\quad \sum_{k=0}^{\lfloor n/2 \rfloor} \binom{2n+2}{n-2k} X_{k};
			\end{equation*}
		\item Forms studied by the authors in the recent paper \cite{Adegoke3}
			\begin{equation*}
			\sum_{k=0}^{\lfloor n/2 \rfloor} \binom{n}{2k} X_{k};
			\end{equation*}
		\item Forms studied by the authors in the recent paper \cite{Adegoke4}
			\begin{equation*}
			\sum_{k=0}^{n} \frac{n}{n+k} \binom{n+k}{n-k} X_{k} \quad \mbox{and}\quad \sum_{k=0}^{n} \frac{k}{n+k} \binom{n+k}{n-k} X_{k}.
			\end{equation*}
		We note that 
			\begin{equation*}
			\sum_{k=0}^{n} \binom{n+k}{2k} X_{k} = \sum_{k=0}^{n} \frac{n}{n+k} \binom{n+k}{n-k} X_{k}+ \sum_{k=0}^{n} \frac{k}{n+k} \binom{n+k}{n-k} X_{k}.
			\end{equation*}
		\end{itemize}
	
	Let $(W_j(a,b;p,q))_{j\geq0}$ be the Horadam sequence 
	\cite{horadam65} defined for all non-negative integers $j$ by the recurrence
		\begin{equation}\label{eq.vhrb5b3}
		W_0  = a,\,\,\,W_1  = b;\quad W_j  = pW_{j - 1}  - qW_{j - 2},\quad j \ge 2,
		\end{equation}
	where $a$, $b$, $p$ and $q$ are arbitrary complex numbers, with $p\ne 0$ and $q\ne 0$. Extension of the definition of $(W_j)$ to negative subscripts is provided by writing the recurrence relation as 
	$$W_{-j}=\frac{1}{q}(pW_{-j+1}-W_{-j+2}).$$
	
	Two important cases of  $(W_j)$ are the Lucas sequences of the first kind, $(U_j(p,q))=(W_j(0,1;p,q))$, and of the second kind, $(V_j(p,q))=(W_j(2,p;p,q))$, so that 
		\begin{equation*}
		U_0=0,\,\,\, U_1=1,\quad U_j=pU_{j-1}-qU_{j-2},\quad j\ge 2;
		\end{equation*}
	and
		\begin{equation*}
		V_0=2,\,\,\, V_1=p,\quad V_j=pV_{j-1}-qV_{j-2}, \quad j\ge 2.
		\end{equation*}
	
	The most well-known Lucas sequences are the Fibonacci sequence $(F_j)=(U_j(1,-1))$ and the sequence of Lucas numbers  $(L_j)=(V_j(1,-1))$.
	
	The Binet formulas for sequences $(U_j)$, $(V_j)$ and $(W_j)$ in the non-degenerate case, $\Delta= p^2 - 4q > 0$, are
		\begin{equation}
		\label{bine.uvw}
		U_j=\frac{\tau^j-\sigma^j}{\sqrt{\Delta}},\qquad V_j=\tau^j+\sigma^j, \qquad
		W_j = A\tau ^j  +  B\sigma ^j\,,
		\end{equation}
	with
	$\displaystyle A=\frac{b - a\sigma}{\sqrt{\Delta}}$ and $B\displaystyle=\frac{a\tau  - b}{\sqrt{\Delta}}$,
	where 
	$$\tau=\tau(p,q)=
	\frac{p+\sqrt{\Delta}}2,\quad \sigma=
	\sigma(p,q)=
	\frac{p-\sqrt{\Delta}}2
	$$
	are the distinct zeros of the characteristic polynomial $x^2-px+q$ of the Horadam sequence \eqref{eq.vhrb5b3}.
	
	The Binet formulas for the Fibonacci and Lucas numbers are
		\begin{equation}
		\label{bine.fl}
		F_j  =  
		\frac{{\alpha ^j  - \beta ^j }}{{\sqrt 5 }},\qquad L_j  = \alpha ^j  + \beta ^j,
		\end{equation}
		where $\alpha=\tau(1,-1)=(1 + \sqrt 5)/2$ is the golden ratio and $\beta=\sigma(1,-1)=-1/\alpha$.
	
	The sequences $(F_n)_{n\geq 0}$ and $(L_n)_{n\geq 0}$ are indexed in the On-Line Encyclopedia of Integer Sequences 
	\cite{OEIS} as entries A000045 and A000032, respectively. For more information on them we recommend the books by Koshy \cite{Koshy} and Vajda \cite{Vajda}, among others. 
	
	In this paper, we introduce relations between binomial sums involving (generalized) Fibonacci and Lucas numbers, and different kinds of binomial coefficients. We also present some relations between sums with two and three binomial coefficients. In the course of exploration we rediscover a few relations presented as problem proposals.
	
	We will make use of the following known results.
	\begin{lemma}\label{lem.jv8c0fd}
		If $a$, $b$, $c$ and $d$ are rational numbers and $\lambda$ is an irrational number, then
			\begin{equation*}
			a + b\,\lambda=c + d\,\lambda\quad \iff\quad a=c,\,\,\, b=d.    
			\end{equation*}
			\end{lemma}
	\begin{lemma}
		For any integer $s$,
			\begin{align}
			&q^s  + \tau ^{2s}  = \tau ^s V_s, \qquad q^s  - \tau ^{2s}  = - \Delta\tau ^s U_s, \label{eq.u9uuagc} \\
			&q^s  + \sigma ^{2s}  = \sigma ^s V_s, \qquad q^s  - \sigma ^{2s}  = \Delta\sigma ^s U_s. \label{eq.dkvfyiz}
			\end{align} 
			In particular,
			\begin{align}
			&( - 1)^s  + \alpha ^{2s}  = \alpha ^s L_s,\qquad  ( - 1)^s  - \alpha ^{2s}  = - \sqrt 5\alpha ^s F_s, \label{eq.ozz3zp6} \\
			&( - 1)^s  + \beta ^{2s}  = \beta ^s L_s,\qquad ( - 1)^s  - \beta ^{2s}  = \sqrt 5\beta ^s F_s \label{eq.syrjcay}.
			\end{align}
			\end{lemma}
	\begin{lemma}\label{lem.ydalnfx}
		Let $r$ and $d$ be any integers. Then 
			\begin{align}
			&V_{r + s}  - \tau ^r V_s  =  - \Delta\sigma ^s  U_r \label{eq.j428hfx},\\ 
			&V_{r + s}  - \sigma ^r V_s  = \Delta \tau ^s U_r \label{eq.cqli6xc},\\
			&U_{r + s}  - \tau ^r U_s  = \sigma ^s U_r\label{eq.iamiky1},\\ 
			&U_{r + s}  - \sigma ^r U_s  = \tau ^s U_r\label{eq.zy0gfyn}. 
			\end{align}
			In particular, \cite{Hoggatt},
			\begin{align}
			&L_{r + s} - L_r \alpha ^s = - \sqrt 5 \beta ^r F_s,\qquad  L_{r + s} - L_r \beta ^s = \sqrt 5 \alpha ^r F_s, \label{es.benssj4} \\ 
			&F_{r + s} - F_r \alpha ^s = \beta ^r F_s,\qquad\qquad\,\, F_{r + s} - F_r \beta ^s = \alpha ^r F_s. \label{eq.pvdw5ja} 
			\end{align}
		\end{lemma}
	\begin{lemma}\label{lem.w65xm59}
		For any integer $j$,
				\begin{align}
			&A\tau ^j  - B\sigma ^j  = \frac{{W_{j + 1}  - qW_{j - 1} }}{\Delta }\label{eq.p951mmh},\\
			&A\sigma ^j  +  B\tau ^j  = q^jW_{-j}\label{eq.trp91mj}.
			\end{align}
		\end{lemma}
	\begin{proof}
		See \cite[Lemma 1]{adegoke21} for a proof of \eqref{eq.p951mmh}. Identity \eqref{eq.trp91mj} is a consequence of the Binet formula.
	\end{proof}

	\section{Relations from a classical polynomial identity}
	
		The first binomial sum relations follow from the next classical polynomial identity which we state in the next lemma.
	\begin{lemma}[{\cite[Identity 6.21]{quaintance}}]\label{lem.v28pabt}
		If $x$ is a complex variable and $m$ and $n$ are non-negative integers, then
		\begin{equation}
			\sum_{k = 0}^n \binom {m - n + k}{k} (1 + x)^{n - k} x^k = \sum_{k = 0}^n \binom{m + 1}{k} x^k \label{eq.wl5p3v8},\quad x\ne -1.
			\end{equation}
	\end{lemma}
	
	According to Gould \cite{quaintance}, identity \eqref{eq.wl5p3v8} 
	is due to Laplace. In addition, we note that the binomial theorem is a special case of \eqref{eq.wl5p3v8} which occurs at $m=n - 1$.
	
	Using
		\begin{equation*}
		\binom {-n}{k} = (-1)^k\binom {n+k-1}{k}
		\end{equation*}
	and replacing $m$ by $m-1$ we have the equivalent and useful form of Lemma \ref{lem.v28pabt}:
		\begin{equation*}
		\sum_{k = 0}^n (-1)^k \binom {n -m}{k} (1 + x)^{n - k} x^k = \sum_{k = 0}^n \binom {m }{k} x^k,\quad x\ne -1.
		\end{equation*}
	\begin{theorem}\label{thm.x9ibwgc}
		If $r$, $s$ and $t$ are any integers and $m$ and $n$ are non-negative integers, then
		\begin{equation*}
			\sum_{k = 0}^n \binom{m - n + k}kU_{r + s}^k U_s^{n - k} W_{t + r(n - k)}  = \sum_{k = 0}^n ( - q^s)^{n - k} \binom{m + 1}{k} U_{r + s}^k U_r^{n - k} W_{t - s(n - k)}.
			\end{equation*}
	\end{theorem}
	\begin{proof}
		Set $x=-U_{r + s}/(U_r\sigma^s)$ in \eqref{eq.wl5p3v8}, use \eqref{eq.zy0gfyn} and multiply through by $\tau^t$, obtaining
			\begin{equation*}
			\sum_{k = 0}^n \binom{m - n + k}{k} U_{r + s}^k U_s^{n - k} \tau^{r(n - k) + t} 
			= (- 1)^t \sum_{k = 0}^n (- 1)^{n-k} \binom{m + 1}{k} U_{r + s}^k U_r^{n - k} \sigma^{s(n - k) - t}.
			\end{equation*}
	
		Similarly, setting $x=-U_{r + s}/(U_r\tau^s)$ in \eqref{eq.wl5p3v8}, using \eqref{eq.iamiky1} and multiplying through by $\sigma^t$, yields
			\begin{equation*}
			\sum_{k = 0}^n \binom{m - n + k}{k} U_{r + s}^k U_s^{n - k} \sigma^{r(n - k) + t}  
			= (- 1)^t \sum_{k = 0}^n (- 1)^{n-k} \binom{m + 1}{k} U_{r + s}^k U_r^{n - k} \tau^{s(n - k) - t}.
			\end{equation*}
		
		The results follow by combining these identities according to the Binet formulas \eqref{bine.uvw} and Lemma \ref{lem.w65xm59}.
	\end{proof}
		In particular,
		\begin{align*}
		&\sum_{k = 0}^n \binom{m - n + k}kU_{r + s}^k U_s^{n - k} V_{r(n - k) + t} = q^t\sum_{k = 0}^n ( - 1)^{n - k}\binom{m + 1}k U_{r + s}^k U_r^{n - k} V_{s(n - k) - t},
		\end{align*}
	and
		\begin{align*}
		&\sum_{k = 0}^n \binom{m - n + k}{k} U_{r + s}^k U_s^{n - k} U_{r(n - k) + t}  = q^{t} \sum_{k = 0}^n ( - 1)^{n - k+1} \binom{m + 1}{k} U_{r + s}^k U_r^{n - k} U_{s(n - k) - t};
		\end{align*}
		with the special cases
			\begin{align}
		&\sum_{k = 0}^n \binom{m - n + k}kF_{r + s}^k F_s^{n - k} L_{r(n - k) + t}  = \sum_{k = 0}^n ( - 1)^{n-k-t} \binom{m + 1}k F_{r + s}^k F_r^{n - k} L_{s(n - k) - t},\label{eq.pj6jm1d}\\
		&\sum_{k = 0}^n \binom{m - n + k}k F_{r + s}^k F_s^{n - k} F_{r(n - k) + t}  =- \sum_{k = 0}^n ( - 1)^{n-k-t}\binom{m + 1}{k} F_{r + s}^k F_r^{n - k} F_{s(n - k) - t}.\label{eq.xgtk3bf}
		\end{align}
		\begin{corollary}
		If $r$, $s$ and $t$ are any integers and $n$ is a non-negative integer, then
		\begin{equation*}
			\sum_{k = 0}^n ( - q^s)^{n - k} \binom{n + 1}{k} U_{r + s}^k U_r^{n - k} W_{t - s(n - k)}  = \sum_{k = 0}^n U_{r + s}^k U_s^{n - k} W_{t + r(n - k)} .
			\end{equation*}
	\end{corollary}
	\begin{proof}
		Set $m=n$ in Theorem \ref{thm.x9ibwgc}.
	\end{proof}
		In particular,
			\begin{align*}
		&\sum_{k = 0}^n (- 1)^{n - k} \binom {n+1}{k}  U_{r + s}^k U_r^{n - k} V_{s(n - k) - t} = \frac{1}{q^t} \sum_{k = 0}^n U_{r + s}^k U_s^{n - k} V_{r(n - k) + t},\\
		&\sum_{k = 0}^n (- 1)^{n - k} \binom {n+1}{k}  U_{r + s}^k U_r^{n - k} U_{s(n - k) - t} = -  \frac{1}{q^t} \sum_{k = 0}^n U_{r + s}^k U_s^{n - k} U_{r(n - k) + t};
		\end{align*}
		with
			\begin{align*}
		&\sum_{k = 0}^n ( - 1)^{n - k} \binom{n + 1}{k}  F_{r + s}^k F_r^{n - k} L_{s(n - k) - t} 
		= (-1)^t \sum_{k = 0}^n F_{r + s}^k F_s^{n - k} L_{r(n - k) + t},\\
		&\sum_{k = 0}^n ( - 1)^{n - k} \binom{n + 1}{k}  F_{r + s}^k F_r^{n - k} F_{s(n - k) - t} 
		= (-1)^{t-1} \sum_{k = 0}^n F_{r + s}^k F_s^{n - k} F_{r(n - k) + t}.
		\end{align*}
		\begin{corollary}\label{corxxx1}
		If $r$, $s$ and $t$ are any integers and $n$ is a non-negative integer, then
				\begin{equation*}
			\sum_{k = 0}^n ( - q^s)^{n - k} \binom{n}{k} U_{r + s}^k U_r^{n - k} W_{t - s(n - k)} 
			= U_s^n W_{t + rn}.
			\end{equation*}
		\end{corollary}
	\begin{proof}
		Set $m=n - 1$ in Theorem \ref{thm.x9ibwgc}.
	\end{proof}
	In particular,
		\begin{align*}
		\sum_{k = 0}^n (- 1)^{n - k} \binom {n}{k} U_{r + s}^k U_r^{n - k} V_{s(n - k) - t} &= \frac{U_s^n V_{rn + t}}{q^t},\\
		\sum_{k = 0}^n  (- 1)^{n - k} \binom {n}{k} U_{r + s}^k U_r^{n - k} U_{s(n - k) - t} &= -  \frac{U_s^n U_{rn + t}}{q^t};
		\end{align*}
	with the special cases
		\begin{align}\label{eq.fuprl19}
		\sum_{k = 0}^n \binom {n}{k} (- 1)^{n - k} F_{r + s}^k F_r^{n - k} L_{s(n - k) - t} &= (-1)^t F_s^n L_{rn + t},\\
		\label{eq.gefkf2e}
		\sum_{k = 0}^n \binom{n}{k} (- 1)^{n - k} F_{r + s}^k F_r^{n - k} F_{s(n - k) - t} &= (-1)^{t + 1}F_s^n F_{rn + t}.
		\end{align}
	
	We mention that identities \eqref{eq.fuprl19} and \eqref{eq.gefkf2e} exhibit strong similarities to those derived by Hoggatt, Phillips and Leonard in \cite{Hoggatt2}.
	\begin{corollary}
		If $m$ and $n$ are non-negative integers and $r$ is any integer, then
				\begin{equation*}
			\sum_{k = 0}^n (-q^r)^k \binom{m - n + k}k W_{r(n - 2k)} = \sum_{k = 0}^n (- q^r)^k \binom{m + 1}{k} V_r^{n - k} W_{-rk}.
			\end{equation*}
		\end{corollary}
	\begin{proof}
		Make the substitutions $r\mapsto 2r$, $s\mapsto -r$ and $t\mapsto -rn$ in Theorem \ref{thm.x9ibwgc} and simplify.
	\end{proof}
	In particular,
		\begin{align*}
		\sum_{k = 0}^n (-q^r)^k \binom{m - n + k}{k}  V_{r(n - 2k)} &= \sum_{k = 0}^n (- 1)^k \binom{m + 1}{k} V_r^{n - k} V_{rk},\\
		\sum_{k = 0}^n (-q^r)^{k } \binom{m - n + k}{k} U_{r(n - 2k)} &= \sum_{k = 0}^n (- 1)^{k + 1} \binom{m + 1}{k} V_r^{n - k} U_{rk};
		\end{align*}
	with the special cases
		\begin{align*}
		&\sum_{k = 0}^n (- 1)^{k(r + 1)} \binom{m - n + k}{k} L_{r(n - 2k)} = \sum_{k = 0}^n (- 1)^k \binom{m + 1}{k} L_r^{n - k} L_{rk},\\
		&\sum_{k = 0}^n (- 1)^{k(r + 1)} \binom{m - n + k}{k} F_{r(n - 2k)} = \sum_{k = 0}^n (- 1)^{k + 1} \binom{m + 1}{k} L_r^{n - k} F_{rk}.
		\end{align*}
		
	By making appropriate substitutions in Theorem \ref{thm.x9ibwgc} many new sum relations can be established. 
	For example, setting $r=1$, $s=-2$, and $t=1$ (or $r=-1$, $s=2$ and $t=-1$) in \eqref{eq.pj6jm1d} gives
		\begin{equation*}
		\sum_{k = 0}^n ( - 1)^k \binom{m - n + k}{k}  L_{n - k + 1} = \sum_{k = 0}^n ( - 1)^k \binom{m + 1}{k} L_{2(n - k) + 1} 
		\end{equation*}
	which at $m=2n$ gives
		\begin{equation}\label{eq.rqep30p}
		\sum_{k = 0}^n ( - 1)^k \binom{n + k}{k} L_{n - k + 1} = \sum_{k = 0}^n ( - 1)^k \binom{2n + 1}{k} L_{2(n - k) + 1}.
		\end{equation}
		
	The corresponding Fibonacci sums from \eqref{eq.xgtk3bf} are of exactly the same structure
		\begin{equation*}
		\sum_{k = 0}^n ( - 1)^k \binom{m - n + k}{k} F_{n - k + 1} = \sum_{k = 0}^n ( - 1)^k \binom{m + 1}{k} F_{2(n - k) + 1} 
		\end{equation*}
		with the special case 
		\begin{equation}\label{eq.rqep30q}
		\sum_{k = 0}^n ( - 1)^k \binom{n + k}{k} F_{n - k + 1} = \sum_{k = 0}^n ( - 1)^k \binom{2n + 1}{k} F_{2(n - k) + 1}.
		\end{equation}
		
	Another example is the relation
		\begin{equation*}
		\sum_{k = 0}^n ( - 1)^k \binom{m - n + k}{k} L_{n - 3k} = \sum_{k = 0}^n ( - 1)^k \binom{m + 1}{k} 2^{n-k} L_{2k} 
		\end{equation*}
		which at $m=2n$ gives
		\begin{equation}
		\sum_{k = 0}^n (-1)^k \binom{n + k}{k}  L_{n - 3k} = \sum_{k = 0}^n (-1)^k \binom{2n + 1}{k} 2^{n-k} L_{2k},
		\end{equation}
		and its Fibonacci counterparts:
		\begin{equation*}
		\sum_{k = 0}^n ( - 1)^k \binom{m - n + k}{k} F_{n - 3k} = \sum_{k = 1}^n ( - 1)^{k-1} \binom{m + 1}{k} 2^{n-k} F_{2k},
		\end{equation*}
		\begin{equation*}
		\sum_{k = 0}^n ( - 1)^k \binom{n + k}{k} F_{n - 3k} = \sum_{k = 1}^n (- 1)^{k-1} \binom{2n + 1}{k}  2^{n-k} F_{2k}.
		\end{equation*}
		\begin{theorem}\label{thm.x9ibwgd}
		If $m$ and $n$ are non-negative integers and $r, s$ are any integers, then
		\begin{equation*}
			\sum_{k = 0}^n ( - q^s)^k \binom{m - n + k}{k} U_r^k U_s^{n - k} W_{r(n - k) - sk + t}   = \sum_{k = 0}^n  ( - q^s)^k \binom{m + 1}{k}  U_r^k U_{r + s}^{n - k} W_{t - sk}.
		\end{equation*}
		\end{theorem}
	\begin{proof}
		Set $x=-U_r\sigma^s/U_{r + s}$ in \eqref{eq.wl5p3v8}, use \eqref{eq.zy0gfyn} and multiply through by $\tau^t$, obtaining
			\begin{equation*}
			\sum_{k = 0}^n (- 1)^k \binom{m - n + k}kq^{sk} U_r^k U_s^{n - k} \tau^{rn - (r + s)k + t} 
			= q^t \sum_{k = 0}^n (- 1)^k \binom{m + 1}{k} U_r^k U_{r + s}^{n - k} \sigma^{sk - t}.
			\end{equation*}
			Similarly, setting $x=-U_r\tau^s/U_{r + s}$ in \eqref{eq.wl5p3v8}, using \eqref{eq.pvdw5ja} and multiplying through by $\sigma^t$, yields
			\begin{equation*}
			\sum_{k = 0}^n (- 1)^k \binom{m - n + k}kq^{sk} U_r^k U_s^{n - k} \sigma^{rn - (r + s)k + t} 
			= q^t \sum_{k = 0}^n (- 1)^k \binom{m + 1}{k} U_r^k U_{r + s}^{n - k} \tau^{sk - t}.
			\end{equation*}
		
			Now, the result follows immediately upon combining according to the Binet formulas~\eqref{bine.uvw}.
	\end{proof}
	
		In particular,
		\begin{align}
		&\sum_{k = 0}^n {( -q^s )^{k} \binom{m - n + k}{k}  U_r^k U_s^{n - k} V_{rn - (r + s)k + t} } = q^{t}\sum_{k = 0}^n ( - 1)^{k} \binom{m + 1}{k} U_r^k U_{r + s}^{n - k} V_{sk - t} ,\label{eq.fijjvrb}\\
		&\sum_{k = 0}^n {( - q^s)^{k} \binom{m - n + k}{k}  U_r^k U_s^{n - k} U_{rn - (r + s)k + t} }  =- q^{t}\sum_{k = 0}^n ( - 1)^{k} \binom{m + 1}{k} U_r^k U_{r + s}^{n - k} U_{sk - t}\label{eq.o39prie};
		\end{align}
		with the special cases 
		\begin{align*}
		\sum_{k = 0}^n {( - 1)^{k(s - 1)} \binom{m - n + k}kF_r^k F_s^{n - k} L_{rn - (r + s)k + t} } &= \sum_{k = 0}^n ( - 1)^{k(s-1)} \binom{m + 1}kF_r^k F_{r + s}^{n - k} L_{t-sk} ,\\
		\sum_{k = 0}^n {( - 1)^{k(s - 1)} \binom{m - n + k}k F_r^k F_s^{n - k} F_{rn - (r + s)k + t} } &= \sum_{k = 0}^n {( - 1)^{k(s-1)} \binom{m + 1}kF_r^k F_{r + s}^{n - k} F_{t-sk} }.
		\end{align*}
		\begin{corollary}
		If $n$ is a non-negative integer and $r$ and $s$ are any integers, then
				\begin{align*}
			&\sum_{k = 0}^n (- 1)^k \binom{n + 1}{k} U_r^k U_{r + s}^{n - k} V_{sk - t}
			= \frac{1}{q^t}\sum_{k = 0}^n (- q^s)^{k} U_r^k U_s^{n - k} V_{rn - (r + s)k + t}, \\
			&\sum_{k = 0}^n (- 1)^k \binom{n + 1}{k} U_r^k U_{r + s}^{n - k} U_{sk - t}
			= -\frac{1}{q^t}\sum_{k = 0}^n (- q^s)^{k} U_r^k U_s^{n - k} U_{rn - (r + s)k + t}.
			\end{align*}
		\end{corollary}
	\begin{proof}
		Set $m=n$ in \eqref{eq.fijjvrb} and \eqref{eq.o39prie}.
	\end{proof}
	In particular,
		\begin{align*}
		&\sum_{k = 0}^n (- 1)^{k(s-1)} \binom{n + 1}{k} F_r^k F_{r + s}^{n - k} L_{t-sk}
		=  \sum_{k = 0}^n (- 1)^{k(s- 1)} F_r^k F_s^{n - k} L_{rn - (r + s)k + t}, \\
		&\sum_{k = 0}^n (- 1)^{k(s-1)} \binom{n + 1}{k} F_r^k F_{r + s}^{n - k} F_{t-sk}
		= \sum_{k = 0}^n (- 1)^{k(s - 1)} F_r^k F_s^{n - k} F_{rn - (r + s)k + t}.
		\end{align*}
		
	We mention that setting $m=n-1$ in Theorem \ref{thm.x9ibwgd} gives again Corollary \ref{corxxx1}.
	\begin{corollary} If $m$ and $n$ are non-negative integers and $r$ is any integer, then
				\begin{align*}
			&\sum_{k = 0}^n \binom{m - n + k}{k} V_r^k V_{r(n - k)} = \sum_{k = 0}^n (- 1)^{n-k} \binom{m + 1}{k} V_r^{k} V_{r(n-k)}, \\
			&\sum_{k = 0}^n \binom{m - n + k}{k} V_r^k U_{r(n - k)} = \sum_{k = 0}^n (- 1)^{n-k+1} \binom{m + 1}{k} V_r^{k} U_{r(n-k)}.
			\end{align*}
		\end{corollary}
	\begin{proof}
		Make the substitutions $r\mapsto 2r$, $s\mapsto -r$, $t\mapsto -rn$ in \eqref{eq.fijjvrb}, \eqref{eq.o39prie} and simplify.
	\end{proof}
	
	In particular,
		\begin{align*}
		&\sum_{k = 0}^n \binom{m - n + k}{k} L_r^k L_{r(n - k)} = \sum_{k = 0}^n (- 1)^{n-k} \binom{m + 1}{k} L_r^{k} L_{r(n-k)}, \\
		&\sum_{k = 0}^n \binom{m - n + k}{k} L_r^k F_{r(n - k)} = \sum_{k = 0}^n (- 1)^{n-k+1} \binom{m + 1}{k} L_r^{k} F_{r(n-k)}.
		\end{align*}
		\begin{theorem}
		If $m$ and $n$ are non-negative integers and $s, t$ are integers, then
		\begin{align*}
			&\sum_{k = 0}^{2n} \binom{m - 2n + k}k2^{2n - k} V_s^k W_{s(2n - k) + t} = W_t \sum_{k = 0}^n \binom{m + 1}{2k} \Delta ^{2(n-k)} U_s^{2(n - k)} V_s^{2k}  \\
			&\qquad\qquad\qquad + \big( {W_{t + 1}  - qW_{t - 1} } \big)\sum_{k = 1}^n \binom{m + 1}{2k - 1} \Delta ^{2 (n - k) } U_s^{2(n - k) + 1} V_s^{2k - 1},\\		
			&\sum_{k = 0}^{2n - 1} \binom{m - 2n + k + 1}k2^{2n - k - 1} V_s^k W_{s(2n - k - 1) + t} = 
			W_t \sum_{k = 1}^n \binom{m + 1}{2k - 1} \Delta ^{2(n - k)} U_s^{2(n - k)} V_s^{2k - 1}  \\
			&\qquad\qquad\qquad +\big( W_{t + 1}  - qW_{t - 1} \big)\sum_{k = 0}^{n - 1} \binom{m + 1}{2k} \Delta ^{2(n-k-1)} U_s^{2(n - k) - 1} V_s^{2k} .
			\end{align*}
	\end{theorem}
	\begin{proof}
		Set $x=V_s/(\Delta U_s)$ in \eqref{eq.wl5p3v8} and multiply through by $\tau^t$ to obtain
			\begin{equation}\label{eq.tbwsjya}
			\begin{split}
			\sum_{k = 0}^n & \binom{m - n + k}k2^{n - k} \tau ^{s(n - k) + t} V_s^k\\ & = \tau ^t \sum_{k = 0}^{\left\lfloor {n/2} \right\rfloor } \binom{m + 1}{2k} \Delta ^{n - 2k} 
			U_s^{n - 2k} V_s^{2k}  + \tau ^t \sum_{k = 1}^{\left\lceil {n/2} \right\rceil } \binom{m + 1}{2k - 1} \Delta ^{n - 2k + 1} U_s^{n - 2k + 1}  V_s^{2k - 1}  .
			\end{split}
			\end{equation}
		
		Similarly, set $x=-V_s/(\Delta U_s)$ in \eqref{eq.wl5p3v8} and multiply through by $\sigma^t$ to obtain
			\begin{equation}\label{eq.n0yvnnm}
			\begin{split}
			&(-1)^n\sum_{k = 0}^n  \binom{m - n + k}k2^{n - k} \sigma ^{s(n - k) + t} V_s^k \\  
			&\qquad = \sigma ^t \sum_{k = 0}^{\left\lfloor {n/2} \right\rfloor } \binom{m + 1}{2k} \Delta ^{n - 2k} U_s^{n - 2k}  V_s^{2k} - \sigma ^t \sum_{k = 1}^{\left\lceil {n/2} \right\rceil } \binom{m + 1}{2k - 1} \Delta ^{n - 2k + 1} U_s^{n - 2k + 1}  V_s^{2k - 1}.
			\end{split}
			\end{equation}
Combine \eqref{eq.tbwsjya} and \eqref{eq.n0yvnnm} according to the Binet formula while making use also of \eqref{eq.p951mmh}. Consider the cases $n\mapsto 2n$ and $n\mapsto 2n - 1$, in turn.
	\end{proof}
	
	In particular,
		\begin{equation}\label{eq.mkt8u24}
		\begin{split}
		&\sum_{k = 0}^{2n} \binom{m - 2n + k}k 2^{2n - k} V_s^k V_{s(2n - k) + t} \\
		&\quad= V_t \sum_{k = 0}^n \binom{m + 1}{2k} \Delta^{2(n-k)} U_s^{2(n-k)} V_s^{2k}  + U_t\sum_{k = 1}^n \binom{m + 1}{2k - 1} \Delta^{2(n-k+1)} U_s^{2(n - k) + 1} V_s^{2k - 1},\\
		\end{split}
		\end{equation}
		\begin{equation}
		\begin{split}
		&\sum_{k = 0}^{2n} {\binom{m - 2n + k}k2^{2n - k} V_s^k U_{s(2n - k) + t} }\\
		&\quad= U_t \sum_{k = 0}^n \binom{m + 1}{2k} \Delta ^{2(n-k)} U_s^{2(n-k)} V_s^{2k}   + V_t \sum_{k = 1}^n \binom{m + 1}{2k - 1} \Delta^{2(n-k)} U_s^{2(n - k) + 1} V_s^{2k - 1} ,
		\end{split}
		\end{equation}
		\begin{equation}
		\begin{split}
		&\sum_{k = 0}^{2n - 1} {\binom{m - 2n + k + 1}k2^{2n - k - 1} V_s^k V_{s(2n - k - 1) + t} }\\
		&\quad  = U_t\sum_{k = 0}^{n - 1} \binom{m + 1}{2k} \Delta ^{2(n - k)} U_s^{2(n - k) - 1} V_s^{2k}   + V_t \sum_{k = 1}^n \binom{m + 1}{2k - 1} \Delta ^{2(n - k)} U_s^{2(n - k)} V_s^{2k - 1} 
		\end{split}
		\end{equation}
		and
		\begin{equation}\label{eq.agyow83}
		\begin{split}
		&\sum_{k = 0}^{2n - 1} {\binom{m - 2n + k + 1}k2^{2n - k - 1} V_s^k U_{s(2n - k - 1) + t} }\\
		&\,\,\,\, = V_t\sum_{k = 0}^{n - 1} \binom{m + 1}{2k} \Delta^{2(n - k - 1)} U_s^{2(n - k) - 1} V_s^{2k}   + U_t \sum_{k = 1}^n \binom{m + 1}{2k - 1} \Delta ^{2(n - k)} U_s^{2(n - k)} V_s^{2k - 1} .
		\end{split}
		\end{equation}
	with the special cases
		\begin{align*}
		&\sum_{k = 0}^{2n} {\binom{m - 2n + k}k2^{2n - k} L_s^k L_{s(2n - k) + t} }  \\ 
		&\qquad= L_t \sum_{k = 0}^n \binom{m + 1}{2k} 5^{n - k} L_s^{2k} F_s^{2(n - k)}   + F_t \sum_{k = 1}^n \binom{m + 1}{2k - 1} 5^{n - k + 1} L_s^{2k - 1} F_s^{2(n - k) + 1},\\
		&\sum_{k = 0}^{2n} {\binom{m - 2n + k}k 2^{2n - k} L_s^k F_{s(2n - k) + t} }\\
		&\qquad= F_t \sum_{k = 0}^n \binom{m + 1}{2k} 5^{n - k} L_s^{2k} F_s^{2(n - k)}  + L_t \sum_{k = 1}^n {\binom{m + 1}{2k - 1} 5^{n - k} L_s^{2k - 1} F_s^{2(n - k) + 1}  },\\
		&\sum_{k = 0}^{2n - 1} {\binom{m - 2n + k+1}{k} 2^{2n - k -1} L_s^k L_{s(2n - k - 1) + t} }\\
		&\qquad = F_t \sum_{k = 0}^{n - 1} \binom{m + 1}{2k} 5^{n - k} L_s^{2k} F_s^{2(n -k)- 1}  
		+ L_t \sum_{k = 1}^n \binom{m + 1}{2k - 1} 5^{n - k} L_s^{2k - 1} F_s^{2(n - k)},\\
		&\sum_{k = 0}^{2n - 1} {\binom{m - 2n + k + 1}k2^{2n - k - 1 } L_s^k F_{s(2n -k - 1) + t} }\\
		&\qquad = L_t \sum_{k = 0}^{n - 1} \binom{m + 1}{2k} 5^{n - k - 1} L_s^{2k} F_s^{2(n - k)-1 }   + F_t \sum_{k = 1}^n \binom{m + 1}{2k - 1} 5^{n - k} 
		L_s^{2k - 1} F_s^{2(n - k)}.
		\end{align*}
	
Note that in \eqref{eq.mkt8u24}--\eqref{eq.agyow83}, we used (see, for example, \cite[Identities (1.16), (1.17)]{Adegoke0})
		\begin{equation}\label{eq.zqa5sbx}
		U_{t + 1} - qU_{t - 1}=V_t,\qquad V_{t + 1} - qV_{t - 1}=\Delta^2U_t.
		\end{equation}
	\begin{lemma}
		If $x$ is a complex variable and $m, n$ are non-negative integers, then
		\begin{equation}\label{eq.g85la9i}
			\sum_{k = 0}^n {\binom{m - n + k}kx^k }  = \sum_{k = 0}^n {\binom{m + 1}kx^k (1 - x)^{n - k}}
			\end{equation}
\end{lemma}
	\begin{proof}
		Use the transformation $\frac{x}{1 + x}\mapsto x$ in \eqref{eq.wl5p3v8}.
	\end{proof}
	\begin{theorem}
		If $m$ and $n$ are non-negative integers and $s, t$ are integers, then
		\begin{equation}\label{eq.nu2jmg5}
			\begin{split}
			&L_t \sum_{k = 0}^n \binom{m - 2n + 2k}{2k}\frac{F_s^{2(n - k)} L_s^{2k}}{5^{k}}   - F_t \sum_{k = 1}^n \binom{m - 2n + 2k - 1}{2k - 1} \frac{F_s^{2(n - k) + 1} L_s^{2k - 1}}{5^{k - 1}} \\
			&\qquad\qquad \qquad\qquad =\Big(\frac{4}{5}\Big)^{n}\sum_{k = 0}^{2n} (-1)^k \binom{m + 1}{k} \frac{L_s^k L_{s(2n - k) + t}}{2^{k}} , 
			\end{split}
			\end{equation}
			\begin{equation}\label{eq.mr53ikq}
			\begin{split}
			&F_t \sum_{k = 0}^n \binom{m - 2n + 2k}{2k} \frac{F_s^{2(n -k)}L_s^{2k}}{5^k}  - L_t \sum_{k = 1}^n \binom{m - 2n + 2k - 1}{2k - 1}\frac{F_s^{2(n - k) + 1} L_s^{2k - 1}}{5^k}  \\
			&\qquad\qquad\qquad\qquad = \Big(\frac{4}{5}\Big)^n\sum_{k = 0}^{2n} ( - 1)^k \binom{m + 1}{k}  \frac{L_s^k F_{s(2n - k) + t}}{2^k}, 
			\end{split}
			\end{equation}
			\begin{equation}\label{eq.zqqebhk}
			\begin{split}
			&L_t \sum_{k = 1}^n \binom{m - 2n + 2k}{2k -1}\frac{F_s^{2(n - k)} L_s^{2k - 1}}{ 5^k}  - F_t \sum_{k = 0}^{n - 1} \binom{m - 2n + 2k + 1}{2k}\frac{F_s^{2(n - k) - 1} L_s^{2k}}{5^k} \\
			&\qquad\qquad \qquad\qquad=\Big(\frac{4}{5}\Big)^n \sum_{k = 0}^{2n - 1} ( - 1)^{k - 1} \binom{m + 1}{k}\frac{L_s^k L_{s(2n - k - 1) + t}}{2^{k+1}}, 
			\end{split}
			\end{equation}
			\begin{equation}\label{eq.n6d13a3}
			\begin{split}
			&F_t \sum_{k = 1}^n \binom{m - 2n + 2k}{2k - 1}\frac{F_s^{2(n - k)} L_s^{2k - 1}} {5^k}  - L_t \sum_{k = 0}^{n - 1} \binom{m - 2n + 2k + 1}{2k}\frac{F_s^{2(n - k) - 1} L_s^{2k}}{5^{k + 1}} \\
			&\qquad\qquad\qquad\qquad = \Big(\frac{4}{5}\Big)^n\sum_{k = 0}^{2n - 1} ( - 1)^{k - 1} \binom{m + 1}{k} \frac{L_s^k F_{s(2n - k - 1) + t}}{2^{k + 1}} . 
			\end{split}
			\end{equation}
	\end{theorem}
	\begin{proof}
		Set $x=L_s/(\sqrt 5 F_s)$ in \eqref{eq.g85la9i} to obtain
		\begin{equation*}
			\sum_{k = 0}^n {\binom{m - n + k}{k}(\sqrt 5 )^{n - k} F_s^{n - k}  L_s^k }  = \sum_{k = 0}^n ( - 1)^{n - k} \binom{m + 1}{k} 2^{n - k} L_s^k  \beta ^{s(n - k)};
			\end{equation*}
		so that
		\begin{equation}\label{eq.hrpmvck}
			\begin{split}
			\sum_{k = 0}^{\left\lfloor {n/2} \right\rfloor } \binom{m - n + 2k}{2k} & \frac{F_s^{n - 2k}  L_s^{2k}}{(\sqrt 5 )^{2k}}  + \sum_{k = 1}^{\left\lceil {n/2} \right\rceil } \binom{m - n + 2k - 1}{2k - 1} \frac{F_s^{n - 2k + 1} L_s^{2k - 1} }{(\sqrt 5 )^{2k - 1} }\\ 
			& = \Big(\frac{2}{\sqrt5}\Big)^n \sum_{k = 0}^n ( - 1)^{n - k} \binom{m + 1}{k} \frac{L_s^k \beta ^{s(n - k)}}{2^k}. 
			\end{split}
			\end{equation}
		
Writing $2n$ for $n$ in \eqref{eq.hrpmvck} (after multiplying through by $\beta^t$) and comparing the coefficients of $\sqrt 5$ produces \eqref{eq.nu2jmg5} and \eqref{eq.mr53ikq}. Writing $2n - 1$ for $n$ gives \eqref{eq.zqqebhk} and \eqref{eq.n6d13a3}.
	\end{proof}
	\begin{corollary}
		If $m$ and $n$ are non-negative integers and $s$ is an integer, then
			\begin{align*}
			&\sum_{k = 0}^n \binom{m - 2n + 2k}{2k} 5^{n - k} F_s^{2(n - k)} L_s^{2k}  = \sum_{k = 0}^{2n} {( - 1)^k \binom{m + 1}k 2^{2n - k - 1} L_s^k L_{s(2n - k) } }, \\
			&\sum_{k = 1}^n \binom{m - 2n + 2k - 1}{2k - 1} 5^{n - k} F_s^{2(n - k) + 1} L_s^{2k - 1} = \sum_{k = 0}^{2n} {( - 1)^{k + 1} \binom{m + 1}k 2^{2n - k - 1} L_s^k F_{s(2n - k) } },\\ 
			&\sum_{k = 1}^n \binom{m - 2n + 2k}{2k - 1} 5^{n - k} F_s^{2(n - k)} L_s^{2k - 1}  = \sum_{k = 0}^{2n - 1} {( - 1)^{k + 1} \binom{m + 1}{k}2^{2n - k - 2} L_s^k L_{s(2n - k - 1)} },\\
			&\sum_{k = 0}^{n - 1} \binom{m - 2n + 2k + 1}{2k} 5^{n - k - 1} F_s^{2(n - k) - 1} L_s^{2k} = \sum_{k = 0}^{2n - 1} {( - 1)^k \binom{m + 1}k2^{2n - k - 2} L_s^k F_{s(2n - k - 1) } }. 
			\end{align*}
	\end{corollary}
	\begin{corollary}
		If $n$ is a non-negative integer and $s$ is any integer, then
		\begin{align*}
			&\sum_{k = 0}^{2n} ( - 1)^k \binom{2n}{k} 2^{2n - 1 - k} L_s^k L_{s(2n - k)}  = 5^n F_s^{2n},\\
			&\sum_{k = 0}^{2n} ( - 1)^{k + 1} \binom{2n + 1}{k} 2^{2n - k - 1} L_s^k F_{s(2n - k) } =\sum_{k = 1}^n 5^{n - k} F_s^{2(n - k) + 1} L_s^{2k - 1},\\
			&\sum_{k = 0}^{2n - 1} {( - 1)^{k + 1} \binom{2n}k2^{2n - k - 2} L_s^k L_{s(2n - k - 1)} }=\sum_{k = 1}^{n} 5^{n - k} F_s^{2(n - k)} L_s^{2k - 1}  ,\\
			&\sum_{k = 0}^{2n - 1} ( - 1)^k \binom{2n}{k} 2^{2n - k - 2} L_s^k F_{s(2n - k - 1) } 
			=\sum_{k = 0}^{n - 1} 5^{n - k - 1} F_s^{2(n - k) - 1} L_s^{2k}.
			\end{align*}
		\end{corollary}
	
	\section{Relations from a recent identity by Alzer}
	
	In 2015 Alzer \cite{Alzer}, building on the work of Aharonov and Elias \cite{Aha}, studied the polynomial 
		\begin{equation}\label{alzer1}
		P_n(x) = (1-x)^{n+1} \sum_{k=0}^n \binom {n+k}{k} x^k, \qquad x\in\mathbb{C}.
		\end{equation}
	Among other things he showed that
	\begin{equation}\label{alzer2}
		P_n(x) = 1 - x + (1-2x) \sum_{k=0}^{n-1} \binom {2k+1}{k} x^{k+1}(1-x)^{k+1}.
		\end{equation}
	
	Such a polynomial identity immediately offers many appealing Fibonacci and Lucas sum relations
	as can been seen from the next series of theorems. 
	\begin{theorem}\label{prop1}
		For each non-negative integer $n$ we have the relations
		\begin{align}\label{eq.qjec5u7}
			&\sum_{k=0}^n (-1)^{k} \binom {n+k}{k} F_{n+1-k} = 1 - 2\sum_{k=0}^{n-1}  (-1)^{k} \binom {2k+1}{k},\\
			\label{eq.qjec5u8}
			&\sum_{k=0}^n (-1)^{k} \binom {n+k}{k} L_{n+1-k} = 1.
			\end{align}
	\end{theorem}
	\begin{proof}
		Set $x=\alpha$ and $x=\beta$ in \eqref{alzer1} and \eqref{alzer2}, respectively, and combine according to the Binet formulas \eqref{bine.fl}.
	\end{proof}
	
	Comparing \eqref{eq.rqep30q} with \eqref{eq.qjec5u7}, and \eqref{eq.rqep30p} with \eqref{eq.qjec5u8}, we find
	\begin{equation*}
		\sum_{k = 0}^n (- 1)^k \binom{2n + 1}{k}  F_{2(n - k) + 1} = 1 - 2\sum_{k=0}^{n-1} (- 1)^k \binom {2k+1}{k} \end{equation*}
and
	\begin{equation*}
		\sum_{k = 0}^n (- 1)^k \binom{2n + 1}{k}  L_{2(n - k) + 1} = 1.
		\end{equation*}
	\begin{theorem}\label{prop2}
		For each non-negative integer $n$ we have the relations
		\begin{align*}
			& \sum_{k=0}^n \binom {n+k}{k} F_{n+2k+1} = (-1)^{n} - \sum_{k=0}^{n-1} (-1)^{n-k} \binom {2k+1}{k}  F_{3(k+2)},\\
			&\sum_{k=0}^n \binom {n+k}{k} L_{n+2k+1} = (-1)^{n} - \sum_{k=0}^{n-1} (-1)^{n-k} \binom {2k+1}{k}  L_{3(k+2)}.
			\end{align*}
	\end{theorem}
	\begin{proof}
		Set $x=\alpha^2$ and $x=\beta^2$ in \eqref{alzer1} and \eqref{alzer2}, respectively, and combine according to the Binet formulas \eqref{bine.fl}.
	\end{proof}
	
	The next theorem generalizes Theorem \ref{prop1}.
	\begin{theorem}\label{prop3}
		For non-negative integers $n$ and $m$ we have the relations
	\begin{align}
			\sum_{k=0}^n (-1)^{mk} &\binom {n+k}{k} \frac{F_{m(n+1-k)}}{L_m^k} 
			= F_m L_m^{n} \Bigg ( 1 + 2\sum_{k=0}^{n-1} \frac{(-1)^{m(k+1)}}{L_m^{2(k+1)}} \binom {2k+1}{k}  \Bigg ),\notag\\
			\label{eq2_prop3}
			&\qquad\sum_{k=0}^n (-1)^{mk} \binom {n+k}{k} \frac{L_{m(n+1-k)}}{L_m^k} = L_m^{n+1}.
			\end{align}
	\end{theorem}
	\begin{proof}
		Set $x=\alpha^m/L_m$ and $x=\beta^m/L_m$ in \eqref{alzer1} and \eqref{alzer2}, respectively, and combine according to the Binet formulas.
	\end{proof}
	
	When $m=1$ then Theorem \ref{prop3} reduces to Theorem \ref{prop1}. As additional examples we state the next relations:
	\begin{equation*}
		\sum_{k=0}^n \binom {n+k}{k} 2^{-k} = 2^{n},
		\end{equation*}
which also appears in Alzer's paper \cite{Alzer} as Eq. (1.4), and
	\begin{gather*}
		\sum_{k=0}^n \binom {n+k}{k} \frac{F_{2(n+1-k)}}{3^k} = 3^n \Bigg ( 1 + 2\sum_{k=0}^{n-1} \binom {2k+1}{k} \frac{1}{9^{k+1}} \Bigg ),\\
		\sum_{k=0}^n \binom {n+k}{k} \frac{L_{2(n+1-k)}}{3^k} = 3^{n+1}.
		\end{gather*}
\begin{theorem}
		For non-negative integer $n$ and any integers $m$ and $t$, we have the relations
			\begin{align*}
			&q^{mn}\sum_{k = 0}^n \binom {n + k}k\frac{W_{mk + t}}{V_m^k}\\
			&\qquad = V_m^n W_{mn + t} - V_m^nU_m\big( W_{m(n + 1) + t + 1}  - qW_{m(n + 1) + t - 1}  \big)\sum_{k = 0}^{n - 1} {\binom {2k + 1}k\frac{{q^{mk} }}{{V_m^{2(k + 1)} }}} .
			\end{align*}
	\end{theorem}
	\begin{proof}
		Set $x=\tau^m/V_m$ and $x=\sigma^m/V_m$ in \eqref{alzer1} and \eqref{alzer2}, respectively, and combine according to the Binet formulas, while making use of Lemma \ref{lem.w65xm59}.
	\end{proof}
	
	In particular,
		\begin{align*}
		&q^{mn}\sum_{k = 0}^n {\binom {n + k}k\frac{{U_{mk + t} }}{V_m^k}}\\
		&\qquad = V_m^n U_{mn + t} - V_m^n U_m \Big( {U_{m(n + 1) + t + 1}  - qU_{m(n + 1) + t - 1} } \Big)\sum_{k = 0}^{n - 1} \binom {2k +1}k \frac{q^{mk}}{{V_m^{2(k + 1)} }},\\
		&q^{mn} \sum_{k = 0}^n {\binom {n + k}k\frac{{V_{mk + t} }}{{V_m^k }}}\\
		&\qquad = V_m^n V_{mn + t} - V_m^n U_m \big( {V_{m(n + 1) + t + 1}  - qV_{m(n + 1) + t - 1} }\big)\sum_{k = 0}^{n - 1} \binom {2k + 1}k\frac{{q^{mk} }}{{V_m^{2(k + 1)} }} ;
		\end{align*}
	with the special cases
		\begin{align*}
		&\sum_{k = 0}^n \binom {n + k}k\frac{F_{mk + t}}{L_m^k }  =  (-1)^{mn}L_m^n F_{mk + t}  - L_m^n F_mL_{m(n + 1) + t} \sum_{k = 0}^{n - 1} {\binom {2k + 1}k\frac{{( - 1)^{m(n-k)} }}{{L_m^{2(k + 1)} }}},\\
		&\sum_{k = 0}^n \binom {n + k}{k}\frac{L_{mk + t} }{{L_m^k }}  = (-1)^{mn} L_m^n L_{mk + t}  - 5L_m^n F_mF_{m(n + 1) + t} \sum_{k = 0}^{n - 1} \binom {2k + 1}k\frac{( - 1)^{m(n-k)} }{L_m^{2(k + 1)}}.
		\end{align*}
	\begin{theorem}\label{prop4}
		For each non-negative integer $n$ we have the relations
			\begin{align*}
			&\sum_{k=0}^n (-1)^{k} \binom {n+k}{k} \frac{F_{2(n+1)-k}}{2^k} 
			= 2^n - \sum_{k=0}^{n-1}  (-1)^{k}\binom {2k+1}{k} 2^{n-2k-1} F_{k+2},\\
			&\sum_{k=0}^n (-1)^{k}\binom {n+k}{k} \frac{L_{2(n+1)-k}}{2^k} 
			= 3\cdot 2^n - \sum_{k=0}^{n-1} (-1)^{k}\binom {2k+1}{k}  2^{n-2k-1} L_{k+2}.
			\end{align*}
			\end{theorem}
	\begin{proof}
		Set $x=\alpha/2$ and $x=\beta/2$ in \eqref{alzer1} and \eqref{alzer2}, respectively, and combine according to the Binet formulas.
	\end{proof}
	\begin{theorem}\label{prop5}
		For each non-negative integer $n$ we have the relations
			\begin{align*}
			&\sum_{k=0}^n (-1)^{k} \binom {n+k}{k}  F_{2(n+1)+k} = 1 - \sum_{k=0}^{n-1} (-1)^{k} \binom {2k+1}{k} F_{3(k+2)},\\
			&\sum_{k=0}^n  (-1)^{k} \binom {n+k}{k} L_{2(n+1)+k} = 3 - \sum_{k=0}^{n-1} (-1)^{k}  \binom {2k+1}{k} L_{3(k+2)}.
			\end{align*}
		\end{theorem}
	\begin{proof}
		Set $x=1/\alpha$ and $x=1/\beta$ in \eqref{alzer1} and \eqref{alzer2}, respectively, and combine according to the Binet formulas.
	\end{proof}
	\begin{remark}
		Combining Theorem \ref{prop2} with Theorem \ref{prop5} gives the relations
					\begin{align*}
			&\sum_{k=0}^n \binom {n+k}{k} F_{n+1+2k} = \sum_{k=0}^{n} (-1)^{n-k}\binom {n+k}{k}  F_{2(n+1)+k},\\
			&\sum_{k=0}^n \binom {n+k}{k} L_{n+1+2k} = 2(-1)^n + \sum_{k=0}^{n}(-1)^{n-k} \binom {n+k}{k} L_{2(n+1)+k}.
			\end{align*}
			\end{remark}
	\begin{theorem}\label{prop6}
	For each non-negative integer $n$ we have the relations
			\begin{align*}
			&3^n\sum_{k=0}^n (-1)^{k} \binom {n+k}{k}  F_{2(n+1+2k)}
			= 1 - \sum_{k=0}^{n-1} (-3)^{k} \binom {2k+1}{k}   \big(3 F_{6k+8} + F_{6k+10}\big),\\
			& 3^n\sum_{k=0}^n (-1)^{k}\binom {n+k}{k}  L_{2(n+1+2k)}
			= 3 - \sum_{k=0}^{n-1} (-3)^{k}\binom {2k+1}{k} \big(3 L_{6k+8} + L_{6k+10}\big).
			\end{align*}
		\end{theorem}
	\begin{proof}
	Set $x=-\alpha^4$ and $x=-\beta^4$ in \eqref{alzer1} and \eqref{alzer2}, respectively, and combine according to the Binet formulas.
	\end{proof}
		The last Theorem in this set involves mixed identities.
	\begin{theorem}\label{prop7}
	For each non-negative integer $n$ we have the relations
			\begin{align*}
			&2^n\sum_{k=0}^n (-1)^{k} \binom {n+k}{k}  F_{2(n+1)+3k}
			= 1 - \sum_{k=0}^{n-1} (-2)^{k} \binom {2k+1}{k} L_{5k+8},\\
			& 2^n \sum_{k=0}^n (-1)^{k} \binom {n+k}{k} L_{2(n+1)+3k}
			= 3 - 5 \sum_{k=0}^{n-1} (-2)^{k} \binom {2k+1}{k} F_{5k+8}.
			\end{align*}
	\end{theorem}
	\begin{proof}
		Set $x=-\alpha^3$ and $x=-\beta^3$ in \eqref{alzer1} and \eqref{alzer2}, respectively, and combine according to the Binet formulas.
	\end{proof}
	
	As a final remark in this section we note that some of the identities presented in this section follow also from the following lemma.
	\begin{lemma}[{\cite[Identities 6.22,\;6.23]{quaintance}}]
		If $x$ is a complex variable and $m$, $n$ are non-negative integers, then
	\begin{align}
			&\sum_{k = 0}^n \binom{n + k}{k} \big( {(1 - x)^{n + 1} x^k + x^{n + 1} (1 - x)^k } \big) = 1, \qquad x\ne 0, \label{eq.eu32lpb} \\
			&\sum_{k = 0}^n \binom{n + k}{k} \big ( (1 - x)^{n + 1} + x^{n + 1 - k} (1 - x)^k \big) = x^{n + 1}, \qquad x\ne 0,\,\, x\ne1 \label{eq.kn9hc8z}.
			\end{align}
	\end{lemma}
	
	For instance, identity \eqref{eq.qjec5u8} is an immediate consequence of \eqref{eq.eu32lpb} at $x=\alpha$. 
	Also, \eqref{eq2_prop3} follows easily from \eqref{eq.eu32lpb}.

	\section{Relations involving two central binomial coefficients}
	
	\begin{lemma}\label{central_bin}
		Let $x$ be a complex variable. Then
	\begin{equation}\label{eq.central_bin}
			\sum_{k = 0}^n \binom {2k}{k} \binom {2(n - k)}{n-k} x^{2k} = \sum_{k = 0}^n \binom {n}{k}^2 (1+x)^{2k} (1-x)^{2(n-k)}.
			\end{equation}
	\end{lemma}
	\begin{proof} From Riordan's book \cite{Riordan} it is known that for the polynomial
	\begin{equation*}
			A_n(t) = \sum_{k = 0}^n \binom {2k}{k} \binom {2(n - k)}{n-k} t^{k}
			\end{equation*}
	we have the relation
	\begin{equation*}
			A_n\big((2t-1)^2\big) = 4^n \sum_{k = 0}^n \binom {n}{k}^2 t^{2k} (1-t)^{2(n-k)}.
			\end{equation*}
	Set $x=2t-1$ and simplify.
	\end{proof}
	\begin{theorem}\label{cbc_prop1}
		For each integer $r$ and each non-negative integer $n$ we have the relations
	\begin{align*}
			&\sum_{k = 0}^n \binom {2k}{k} \binom {2(n - k)}{n-k} F_{2k+r} = \sum_{k = 0}^n \binom {n}{k}^2 F_{6k-2n+r},\\
			&\sum_{k = 0}^n \binom {2k}{k} \binom {2(n - k)}{n-k} L_{2k+r} = \sum_{k = 0}^n \binom {n}{k}^2 L_{6k-2n+r}.
			\end{align*}
	\end{theorem}
	\begin{proof}
		Set $x=\alpha$ and $x=\beta$ in Lemma \ref{central_bin}, multiply through by $\alpha^r$ and $\beta^r$, respectively, 
		and combine according to the Binet formulas \eqref{bine.fl}.
	\end{proof}
	\begin{theorem}\label{cbc_prop2}
		For each integer $r$ and each non-negative integer $n$ we have the relations
	\begin{align*}
			&\sum_{k = 0}^n \binom {2k}{k} \binom {2(n - k)}{n-k} F_{4k+r} = F_{2n+r} \sum_{k = 0}^n \binom {n}{k}^2 5^k,\\
			&\sum_{k = 0}^n \binom {2k}{k} \binom {2(n - k)}{n-k} L_{4k+r} = L_{2n+r} \sum_{k = 0}^n \binom {n}{k}^2 5^k.
			\end{align*}
\end{theorem}
	\begin{proof}
		Set $x=\alpha^2$ and $x=\beta^2$ in Lemma \ref{central_bin}, multiply through by $\alpha^r$ and $\beta^r$, respectively, 
		and combine according to the Binet formulas.
	\end{proof}
	
	We note the following particular results:
	\begin{align*}
		&\sum_{k = 0}^n \binom {2k}{k} \binom {2(n - k)}{n-k} F_{2(2k-n)} = 0,\\
		&\sum_{k = 0}^n \binom {2k}{k} \binom {2(n - k)}{n-k} L_{2(2k-n)} = 2\sum_{k = 0}^n \binom {n}{k}^2 5^k.
		\end{align*}
\begin{theorem}\label{cbc_prop3}
		For each integer $r$ and each non-negative integer $n$ we have the relations
	\begin{align*}
			&\sum_{k = 0}^n \binom {2k}{k} \binom {2(n - k)}{n-k} 4^{n-k} F_{2k+r} =  \sum_{k = 0}^n \binom {n}{k}^2 5^k F_{6k-4n+r},\\
			&\sum_{k = 0}^n \binom {2k}{k} \binom {2(n - k)}{n-k} 4^{n-k} L_{2k+r} =  \sum_{k = 0}^n \binom {n}{k}^2 5^k L_{6k-4n+r}.
			\end{align*}
\end{theorem}
	\begin{proof}
		Set $x=\alpha/2$ and $x=\beta/2$ in Lemma \ref{central_bin}, multiply through by $\alpha^r$ and $\beta^r$, respectively, 
		and combine according to the Binet formulas.
	\end{proof}
	\begin{theorem}\label{cbc_prop4}
		For each integer $r$ and each non-negative integer $n$ we have the relations
		\begin{align*}
			&F_r \sum_{k = 0}^n \binom {2k}{k} \binom {2(n - k)}{n-k} 5^k 4^{n-k}  =  \sum_{k = 0}^n \binom {n}{k}^2 F_{6(2k-n)+r},\\
			&L_r \sum_{k = 0}^n \binom {2k}{k} \binom {2(n - k)}{n-k} 5^k 4^{n-k} = \sum_{k = 0}^n \binom {n}{k}^2 L_{6(2k-n)+r}.
		\end{align*}
\end{theorem}
	\begin{proof}
		Set $x=\sqrt{5}/2$ and $x=-\sqrt{5}/2$ in Lemma \ref{central_bin}, multiply through by $\alpha^r$ and $\beta^r$, respectively, 
		and combine according to the Binet formulas.
	\end{proof}
	\begin{theorem}\label{cbc_prop5}
		For each integer $r$ and each non-negative integer $n$ we have the relations
	\begin{align*}
			&\sum_{k = 0}^n \binom {2k}{k} \binom {2(n - k)}{n-k} F_{6k+r}  =  4^{n} \sum_{k = 0}^n \binom {n}{k}^2 F_{2(n+k)+r},\\
			&\sum_{k = 0}^n \binom {2k}{k} \binom {2(n - k)}{n-k} L_{6k+r}  =  4^{n} \sum_{k = 0}^n \binom {n}{k}^2 L_{2(n+k)+r}.
			\end{align*}
\end{theorem}
	\begin{proof}
		Set $x=\alpha^3$ and $x=\beta^3$ in Lemma \ref{central_bin}, multiply through by $\alpha^r$ and $\beta^r$, respectively, 
		and combine according to the Binet formulas.
	\end{proof}
	\begin{theorem}\label{cbc_prop6}
		For each integer $r$ and each non-negative integer 
		$n$ we have the relations
	\begin{align*}
			&\sum_{k = 0}^n \binom {2k}{k} \binom {2(n - k)}{n-k} F_{8k+r} = F_{4n+r} \sum_{k = 0}^n \binom {n}{k}^2 9^k 5^{n-k},\\
			&\sum_{k = 0}^n \binom {2k}{k} \binom {2(n - k)}{n-k} L_{8k+r} = L_{4n+r} \sum_{k = 0}^n \binom {n}{k}^2 9^k 5^{n-k}.
			\end{align*}
\end{theorem}
	\begin{proof}
		Set $x=\alpha^4$ and $x=\beta^4$ in Lemma \ref{central_bin}, multiply through by $\alpha^r$ and $\beta^r$, respectively, 
		and combine according to the Binet formulas.
	\end{proof}
	
	In particular,
	\begin{equation*}
		\sum_{k = 0}^n \binom {2k}{k} \binom {2(n - k)}{n-k} F_{4(2k-n)} = 0
		\end{equation*}
and 
		\begin{equation*}
		\sum_{k = 0}^n \binom {2k}{k} \binom {2(n - k)}{n-k} L_{4(2k-n)} =  2 \sum_{k = 0}^n \binom {n}{k}^2 9^k 5^{n-k}.
		\end{equation*}
	
	We proceed with some identities involving an additional parameter.
	\begin{theorem}
		If $n$ is a non-negative integer and $r, s$ are any integers, then
			\begin{equation*}
			\sum_{k = 0}^n \binom{2k}k\binom{2(n - k)}{n - k}q^{2s(n - k)} W_{4ks + r}  = W_{2ns + r} \sum_{k = 0}^n \binom nk^2 \Delta ^{2k} V_s^{2(n - k)} U_s^{2k} .
			\end{equation*}
	\end{theorem}
	\begin{proof}
		Set $x=\sigma^s/\tau^s$ and $x=\tau^s/\sigma^s$, in turn, in Lemma \ref{central_bin}, multiply through by $\tau^r$ and $\sigma^r$, respectively, and combine according to the Binet formulas.
	\end{proof}
	
	In particular,
		\begin{align*}
		&\sum_{k = 0}^n {\binom{2k}k\binom{2(n - k)}{n - k}q^{2s(n - k)} U_{4ks + r} }  = U_{2ns + r} \sum_{k = 0}^n \binom nk^2 \Delta ^{2k} V_s^{2(n - k)} U_s^{2k},\\
		&\sum_{k = 0}^n {\binom{2k}k\binom{2(n - k)}{n - k}q^{2s(n - k)} V_{4ks + r} }  = V_{2ns + r} \sum_{k = 0}^n \binom nk^2 \Delta ^{2k} V_s^{2(n - k)} U_s^{2k};
		\end{align*}
with the special cases
		\begin{align}\label{eq.ki6av2j}
		\sum_{k = 0}^n \binom {2k}{k} \binom {2(n - k)}{n-k} F_{4sk+r} 
		= F_{2ns+r} \sum_{k = 0}^n \binom {n}{k}^2(5F_s^2)^{n-k} L_s^{2k},\\
		\label{eq.hkxzb8j}
		\sum_{k = 0}^n \binom {2k}{k} \binom {2(n - k)}{n-k} L_{4sk+r} 
		= L_{2ns+r} \sum_{k = 0}^n \binom {n}{k}^2 (5F_s^2)^{n-k} L_s^{2k}.
		\end{align}
		\begin{remark}
		Note that Theorems \ref{cbc_prop2} and \ref{cbc_prop6} are particular cases of \eqref{eq.ki6av2j} and \eqref{eq.hkxzb8j} at $s=1$ and $s=2$, respectively.
	\end{remark}
	\begin{theorem}\label{cbc_prop8}
		For integers $r$ and $s\geq 1$, and each non-negative integer $n$ we have the  relations
		
			\begin{align*}
			&F_r \sum_{k = 0}^n \binom {2k}{k} \binom {2(n - k)}{n-k} \big (5 F_s^2\big )^k L_s^{2(n-k)} 
			= 4^{n}\sum_{k = 0}^n \binom {n}{k}^2 F_{2s(2k-n)+r},\\
			&L_r \sum_{k = 0}^n \binom {2k}{k} \binom {2(n - k)}{n-k} \big (5 F_s^2\big)^k L_s^{2(n-k)} 
			= 4^{n} \sum_{k = 0}^n \binom {n}{k}^2 L_{2s(2k-n)+r}.
			\end{align*}
	\end{theorem}
	\begin{proof}
		Set $x=\sqrt{5}F_s/L_s$ and $x=-\sqrt{5}F_s/L_s$ in Lemma \ref{central_bin}, multiply through by $\alpha^r$ and $\beta^r$, respectively, 
		and combine according to the Binet formulas.
	\end{proof}
\begin{theorem}\label{cbc_prop9}
		For integers $r$ and $s\geq 1$, and each non-negative integer $n$ we have the relations
	\begin{align*}
			&F_r \sum_{k = 0}^n \binom {2k}{k} \binom {2(n - k)}{n-k}  (5F_s^2)^{n-k} L_s^{2k} 
			= 4^n\sum_{k = 0}^n \binom {n}{k}^2 F_{2s(2k-n)+r},\\
			&L_r \sum_{k = 0}^n \binom {2k}{k} \binom {2(n - k)}{n-k}  (5F_s^2)^{n-k} L_s^{2k} 
			= 4^{n} \sum_{k = 0}^n \binom {n}{k}^2 L_{2s(2k-n)+r}.
			\end{align*}
	\end{theorem}
	\begin{proof}
		Set $x=L_s/(\sqrt{5}F_s)$ and $x=-L_s/(\sqrt{5}F_s)$ in Lemma \ref{central_bin}, multiply through by $\alpha^r$ and $\beta^r$, respectively, 
		and combine according to the Binet formulas.
	\end{proof}
	\begin{theorem}\label{cbc_prop10}
		For each integer $r$ and each non-negative integer $n$ we have the relations
			\begin{align*}
			&\sum_{k = 0}^n \binom {2k}k \binom {2(n - k)}{n-k}5^{n - k} F_{6k + r}  = 4^n \sum_{k = 0}^n \binom nk^24^k F_{2k + r},\\
			&\sum_{k = 0}^n \binom {2k}k \binom {2(n - k)}{n-k}5^{n - k} L_{6k + r}  = 4^n \sum_{k = 0}^n {\binom nk^24^k L_{2k + r} }.
			\end{align*}
	\end{theorem}
	\begin{proof}
		Set $x=\alpha^3/\sqrt 5$ and $x=\beta^3/\sqrt 5$, in turn, in Lemma \ref{central_bin}, multiply through by $\alpha^r$ and $\beta^r$, respectively, and combine according to the Binet formulas, using also the fact that
		$\sqrt 5 - \alpha^3=-2$ and $\sqrt 5 + \alpha^3 =4\alpha$.
	\end{proof}
	\begin{theorem}\label{cbc_prop11}
		For each integer $r$, and each non-negative integer $n$ we have the relations
			\begin{align*}
			&\sum_{k = 0}^n \binom {2k}k \binom {2(n - k)}{n-k}9^{n - k} F_{6k + r}  = 4^n \sum_{k = 0}^n \binom nk^2 5^k F_{4k -2n + r},\\
			&\sum_{k = 0}^n \binom {2k}k \binom {2(n - k)}{n-k}9^{n - k} L_{6k + r}  = 4^n \sum_{k = 0}^n {\binom nk^2 5^k L_{4k -2n + r} }.
			\end{align*}
		\end{theorem}
	\begin{proof}
		Set $x=\alpha^3/3$ and $x=\beta^3/3$, in turn, in Lemma \ref{central_bin}, multiply through by $\alpha^r$ and $\beta^r$, 
		respectively, and combine according to the Binet formulas, using also the fact that
		$3 - \alpha^3 = 2\beta$ and $3 + \alpha^3 = 2\sqrt 5\alpha$.
	\end{proof}
	\begin{theorem}
		If $n$ is a non-negative integer and $r, s$ are any integers, then
			\begin{equation*}
			\sum_{k = 0}^n \binom nk^2 q^{2s(n-k)} W_{4sk + r} = \frac{W_{2sn + r}}{4^n}\sum_{k = 0}^n \binom{2k}{k}\binom{2(n - k)}{n - k} \Delta ^{2k} V_s^{2(n - k)} U_s^{2k} .
			\end{equation*}
			\end{theorem}
	\begin{proof}
		Set $x=\Delta U_s/V_s$ in Lemma \ref{central_bin} and multiply through by $\sigma^r$. Repeat for $x=-\Delta U_s/V_s$ and multiply through by $\tau^r$. Now combine the resulting equations using the Binet formula.
	\end{proof}
	
	In particular,
		\begin{align*}
		&\sum_{k = 0}^n \binom nk^2 q^{2s(n-k)} U_{4sk + r}  = \frac{{U_{2sn + r} }}{4^n}\sum_{k = 0}^n \binom{2k}k\binom{2(n - k)}{n - k} \Delta ^{2k} V_s^{2(n - k)} U_s^{2k},\\
		&\sum_{k = 0}^n \binom nk^2 q^{2s(n-k)} V_{4sk + r}  = \frac{V_{2sn + r}}{4^n}\sum_{k = 0}^n \binom{2k}{k}\binom{2(n - k)}{n - k}\Delta ^{2k} V_s^{2(n - k)} U_s^{2k};
		\end{align*}
with the special cases
		\begin{align*}
		&\sum_{k = 0}^n \binom nk^2 F_{4sk + r}  = \frac{F_{2sn + r}}{4^n}\sum_{k = 0}^n \binom{2k}k\binom{2(n - k)}{n - k} 5^k L_s^{2(n - k)} F_s^{2k}  ,\\
		&\sum_{k = 0}^n \binom nk^2L_{4sk + r} = \frac{L_{2sn + r} }{{4^n}}\sum_{k = 0}^n \binom{2k}k\binom{2(n - k)}{n - k} 5^k L_s^{2(n - k)} F_s^{2k}.
		\end{align*}

	\section{Another class of identities with squared binomial \\coefficients}
	
	\begin{lemma}[\cite{Elementary}]\label{sq_bin}
		If $n$ is a non-negative integer and $x$ is any complex variable, then
			\begin{equation}\label{eq.sq_bin}
			\sum_{k=0}^n \binom {n}{k}^2 x^k = \sum_{k=0}^n \binom {n}{k} \binom {n+k}{k} (x-1)^{n-k}.
			\end{equation}
			\end{lemma}
\begin{theorem}\label{thm_sq_bin}
		Let $r$, $s$ and $m$ be arbitrary integers with $r\neq 0$. Then for each non-negative integer $n$ we have the relations
				\begin{equation}\label{sqbin_Fib}
			\begin{split}
			\sum_{k = 0}^n (-1)^{k(s+1)}& \binom {n}{k}^2  \Big (\frac{F_s}{F_r}\Big )^k  F_{(r+s)k+m}\\
			&= \sum_{k = 0}^n (-1)^{(s+1)(n-k)} \binom {n}{k} \binom {n+k}{k}  \Big (\frac{F_{r+s}}{F_r}\Big )^{n-k}  F_{s(n-k)+m},
			\end{split}
			\end{equation}
			\begin{equation}
			\label{sqbin_Luc}
			\begin{split}
			\sum_{k = 0}^n (-1)^{k(s+1)}& \binom {n}{k}^2 \Big (\frac{F_s}{F_r}\Big )^k  L_{(r+s)k+m}\\
			&= \sum_{k = 0}^n (-1)^{(s+1)(n-k)} \binom {n}{k} \binom {n+k}{k} \Big (\frac{F_{r+s}}{F_r}\Big )^{n-k}  L_{s(n-k)+m}.
			\end{split}
			\end{equation}
			\end{theorem}
	\begin{proof} 
		Set $x=-F_{s}\beta^r/(F_r \alpha^{s})$ and $x=-F_{s}\alpha^r/(F_r \beta^{s})$, respectively, in Lemma \ref{sq_bin}, 
		and use Lemma \ref{lem.ydalnfx}. Multiply through by $\alpha^m$ and $\beta^m$, respectively, and combine according to the Binet formulas.
	\end{proof}
		\begin{corollary}\label{cor1_sqbin}
		For each integer $m$ and each non-negative integer $n$ we have 
			\begin{align*}
			&\sum_{k = 0}^n \binom {n}{k}^2 F_{k+m} = (-1)^{m+1}\sum_{k = 0}^n (-1)^{n-k} \binom {n}{k} \binom {n+k}{k}  F_{n-k-m},\\
			&\sum_{k = 0}^n \binom {n}{k}^2 L_{k+m} = (-1)^{m}\sum_{k = 0}^n (-1)^{n-k}\binom {n}{k} \binom {n+k}{k} L_{n-k-m}.
			\end{align*}
			\end{corollary}
	\begin{proof}
		Set $r=2$ and $s=-1$ in Theorem \ref{thm_sq_bin}.
	\end{proof}
	
	The case $m=0$ in  Corollary \ref{cor1_sqbin} was proposed by Carlitz as a problem in the Fibonacci Quarterly \cite{Carlitz_AP1} (with a typo).
	\begin{corollary}\label{cor2_sqbin}
		For each integer $m$ and each non-negative integer $n$ we have 
			\begin{align*}
			\sum_{k = 0}^n \binom {n}{k}^2 F_{2k+m} &= \sum_{k = 0}^n \binom {n}{k} \binom {n+k}{k} F_{n-k+m},\\
			\sum_{k = 0}^n \binom {n}{k}^2 L_{2k+m} &= \sum_{k = 0}^n \binom {n}{k} \binom {n+k}{k} L_{n-k+m}.
			\end{align*}
		\end{corollary}
	\begin{proof}
		Set $r=s=1$ in Theorem \ref{thm_sq_bin}.
	\end{proof}
	
	The case $m=0$ in  Corollary \ref{cor2_sqbin} was proposed by Carlitz as another problem in the Fibonacci Quarterly \cite{Carlitz_AP2}.
	\begin{corollary}
		For each integer $m$ and each non-negative integer $n$ we have 
			\begin{align*}
			&\sum_{k = 0}^n \binom {n}{k}^2 F_{3k+m} = \sum_{k = 0}^n \binom {n}{k} \binom {n+k}{k} 2^{n-k} F_{n-k+m},\\
			&\sum_{k = 0}^n \binom {n}{k}^2 L_{3k+m} = \sum_{k = 0}^n \binom {n}{k} \binom {n+k}{k} 2^{n-k} L_{n-k+m}.
			\end{align*}
		\end{corollary}
	\begin{proof}
		Set $r=2$ and $s=1$ in Theorem \ref{thm_sq_bin}.
	\end{proof}
	\begin{corollary}
		For each integer $m$ and each non-negative integer $n$ we have 
			\begin{align*}
			&\sum_{k = 0}^n (-1)^k \binom {n}{k}^2  F_{k+m} = \sum_{k = 0}^n (-1)^{n-k} 
			\binom {n}{k} \binom {n+k}{k} F_{2(n-k)+m},\\
			&\sum_{k = 0}^n (-1)^k \binom {n}{k}^2  L_{k+m} = \sum_{k = 0}^n (-1)^{n-k} \binom {n}{k} \binom {n+k}{k} L_{2(n-k)+m}.
			\end{align*}
	\end{corollary}
	\begin{proof} Set $r=-1$ and $s=2$ in Theorem \ref{thm_sq_bin}.
	\end{proof}
	\begin{corollary}
		For each integer $r$ and each non-negative integer $n$ we have 
	\begin{align*}
			\sum_{k = 0}^n (-1)^k\binom {n}{k}^2  F_{3k+m} &= \sum_{k = 0}^n (-2)^{n-k}\binom {n}{k} \binom {n+k}{k} F_{2(n-k)+m},\\
			\sum_{k = 0}^n (-1)^k \binom {n}{k}^2  L_{3k+m} &= \sum_{k = 0}^n (-2)^{n-k} \binom {n}{k} \binom {n+k}{k}  L_{2(n-k)+m}.
			\end{align*}
	\end{corollary}
	\begin{proof}
		Set $r=1$ and $s=2$ in Theorem \ref{thm_sq_bin}.
	\end{proof}
	\begin{corollary}
		For each integer $r$ and each non-negative integer $n$ we have 
	\begin{align*}
			\sum_{k = 0}^n (-1)^{(r+1)k}\binom {n}{k}^2  F_{2rk+m} 
			&= \sum_{k = 0}^n (-1)^{(r+1)(n-k)} \binom {n}{k} \binom {n+k}{k}  L_{r}^{n-k} F_{r(n-k)+m},\\
			\sum_{k = 0}^n (-1)^{(r+1)k}  \binom {n}{k}^2 L_{2rk+m} 
			&= \sum_{k = 0}^n (-1)^{(r+1)(n-k)}  \binom {n}{k} \binom {n+k}{k} L_{r}^{n-k} L_{r(n-k)+m}.
			\end{align*}
	\end{corollary}
	\begin{proof}
		Set $s=r$ in Theorem \ref{thm_sq_bin}.
	\end{proof}
	\begin{corollary}
		For each integer $r$ and each non-negative integer $n$ we have 
	\begin{align*}
			\sum_{k = 0}^n (-1)^k \binom {n}{k}^2  L_r^k F_{3rk+m} 
			&= \sum_{k = 0}^n (-1)^{n-k} \binom {n}{k} \binom {n+k}{k}  \big(L_{2r}+(-1)^r\big)^{n-k} F_{2r(n-k)+m},\\
			\sum_{k = 0}^n (-1)^k \binom {n}{k}^2  L_r^k L_{3rk+m} 
			&= \sum_{k = 0}^n (-1)^{n-k} \binom {n}{k} \binom {n+k}{k}  \big(L_{2r}+(-1)^r\big)^{n-k} L_{2r(n-k)+m}.
			\end{align*}
	\end{corollary}
	\begin{proof}
		Set $s=2r$ in Theorem \ref{thm_sq_bin} and make use of $F_{3r}/F_{r} = L_{2r} + (-1)^r$.
	\end{proof}
\begin{theorem}\label{thm2_sq_bin}
		For each non-negative integers $r$, $m$ and $n$ we have the relations
		\begin{align*}
			\sum_{k = 0}^n \binom {n}{k}^2 L_m^{n-k} F_{mk+r} 
			&= (-1)^{r-1}\sum_{k = 0}^n (-1)^{n-k} \binom {n}{k} \binom {n+k}{k}  L_m^{k}F_{m(n-k)-r},\\
			\sum_{k = 0}^n \binom {n}{k}^2 L_m^{n-k} L_{mk+r} 
			&= (-1)^{r}\sum_{k = 0}^n (-1)^{n-k} \binom {n}{k} \binom {n+k}{k} L_m^{k} L_{m(n-k)-r}.
			\end{align*}
		\end{theorem}
	\begin{proof}
		Set $x=\alpha^m/L_m$ and $x=\beta^m/L_m$ in Lemma \ref{sq_bin}, multiply through by $\alpha^r$ and $\beta^r$, respectively, 
		and combine according to the Binet formulas.
	\end{proof}
	
	When $m=1$ then Theorem \ref{thm2_sq_bin} reduces to Corollary \ref{cor1_sqbin}. When $m=2$ then
		\begin{align*}
		&\sum_{k = 0}^n \binom {n}{k}^2 3^{n-k} F_{2k+r} 
		= (-1)^{n-r-1}\sum_{k = 0}^n(-3)^{k} \binom {n}{k} \binom {n+k}{k}   F_{2(n-k)-r},\\
		&\sum_{k = 0}^n \binom {n}{k}^2  3^{n-k} L_{2k+r} 
		= (-1)^{n-r}\sum_{k = 0}^n (-3)^{k} \binom {n}{k} \binom {n+k}{k}   L_{2(n-k)-r}.
		\end{align*}
	\begin{theorem}\label{thm3_sq_bin}
		For each non-negative integer $n$, any odd integer $m$ and any integer $r$ we have the relations
	\begin{align*}
			\sum_{k = 0}^n \binom {n}{k}^2 F_{2mk+r} &= \sum_{k = 0}^n \binom {n}{k} \binom {n+k}{k} L_m^{n-k}F_{m(n-k)+r},\\
			\sum_{k = 0}^n \binom {n}{k}^2 L_{2mk+r} &= \sum_{k = 0}^n \binom {n}{k} \binom {n+k}{k} L_m^{n-k}L_{m(n-k)+r}.
			\end{align*}
		\end{theorem}
	\begin{proof}
		Set $x=\alpha^{2m}$ and $x=\beta^{2m}$, $m$ odd, in Lemma \ref{sq_bin}, and use the facts that
		$\alpha^{2m}-1 = \alpha^m L_m$  and $\beta^{2m}-1 = \beta^m L_m$.
		Multiply through by $\alpha^r$ and $\beta^r$, respectively, and combine according to the Binet formulas.
	\end{proof}
	
	When $m=1$ then Theorem \ref{thm3_sq_bin} reduces to Corollary \ref{cor2_sqbin}.
	
	We conclude this section with a sort of inverse relations compared to those from Theorem \ref{thm_sq_bin}.
	\begin{theorem}\label{thm4_sq_bin}
		For each non-negative integer $n$ and any integers $m$, $r$ and $s$, we have the relations
		\begin{align*}
			\sum_{k = 0}^n (-1)^{sk} \binom {n}{k}^2 \Big  (\frac{F_{r+s}}{F_{r}}\Big )^k  F_{sk+m} 
			&= \sum_{k = 0}^n (-1)^{s(n-k)} \binom {n}{k} \binom {n+k}{k} \Big (\frac{F_{s}}{F_{r}}\Big )^{n-k}  F_{(r+s)(n-k)+m},\\ 
			\sum_{k = 0}^n (-1)^{sk} \binom {n}{k}^2 \Big (\frac{F_{r+s}}{F_{r}}\Big )^k  L_{sk+m} 
			&= \sum_{k = 0}^n (-1)^{s(n-k)} \binom {n}{k} \binom {n+k}{k} \Big (\frac{F_{s}}{F_{r}}\Big )^{n-k}  L_{(r+s)(n-k)+m}.
			\end{align*}
	\end{theorem}
	\begin{proof}
		Set $x=F_{r+s}/(\alpha^{s}F_r)$ and $x=F_{r+s}/(\beta^{s}F_r)$, respectively, in Lemma \ref{sq_bin}, and use Lemma~\ref{lem.ydalnfx}. 
		Multiply through by $\alpha^m$ and $\beta^m$, respectively, and combine according to the Binet formulas.
	\end{proof}

	\section{More identities with two binomial coefficients}
	
	\begin{lemma}[{\cite[Identity (3.17)]{Gould2}}]
		If $n$ is a non-negative integer, $m$ is any real number and $x$ is any complex variable, then
	\begin{equation}\label{eq.iz4pi7f}
			\sum_{k=0}^n \binom {n}{k} \binom {m + n - k}{n} (x - 1)^k = \sum_{k=0}^n \binom {n}{k} \binom {m}{k} x^k.
			\end{equation}
	In particular, 
	\begin{equation}\label{eq.sq_bin2}
			\sum_{k=0}^n \binom {n}{k} \binom {2n - k}{n} (x - 1)^k = \sum_{k=0}^n \binom {n}{k}^2 x^k
			\end{equation}
			\begin{equation}
			\sum_{k = 0}^n \frac{ \binom {n}{k} \binom {2(2n - k)}{2n - k} \binom {2n - k}{n}}{\binom {2(n - k)}{n - k}} (x-1)^k 
			= \binom {2n}{n} \sum_{k = 0}^n \frac{\binom {n}{k}^2}{\binom {2(n - k)}{n - k}} 2^{2(n-k)} x^k.
			\end{equation}
		\end{lemma}
	\begin{proof}
		The first particular case is obvious. The second follows upon setting $m = n - 1/2$ in \eqref{eq.iz4pi7f} 
		and using 
		$$\binom {n - 1/2}{k} = \frac{\binom{2n}{n} \binom {n}{k}}{2^{2k} \binom {2(n - k)}{n - k}}$$ with $0\le k\le n$ (see \cite[Identity (Z.45)]{Gould2}).
	\end{proof}
\begin{remark}
		Comparing \eqref{eq.sq_bin} with \eqref{eq.sq_bin2} we immediately get an ``identity'' of the form
		\begin{equation*}
			\sum_{k=0}^n \binom {n}{k} \binom {2n - k}{n} (x - 1)^k = \sum_{k=0}^n \binom {n}{k} \binom {n + k}{k} (x - 1)^{n-k}.
			\end{equation*}
	
		Such an identity does not contain any new information as the identities can be trivially transformed into each other by reindexing
		$\sum_{k=0}^n a_k = \sum_{k=0}^n a_{n-k}$. This shows that the binomial sum relations from the previous section are actually special instances of those derived now.
	\end{remark}
\begin{theorem}\label{thm_two_bin_gen}
		Let $r,s$ and $p$ be arbitrary integers with $r\neq 0$, and let $m$ is any real number. Then for all non-negative integer $n$ we have the relations
		\begin{equation}\label{twobin_Fib}
			\begin{split}
			\sum_{k = 0}^n (-1)^{(s+1)k} \binom {n}{k} \binom {m + n - k}{n} & \Big (\frac{F_{r+s}}{F_r}\Big )^k  F_{sk+p}\\
			&= \sum_{k = 0}^n (-1)^{(s+1)k} \binom {n}{k} \binom {m}{k} \Big (\frac{F_{s}}{F_r}\Big )^k  F_{(r+s)k+p},
			\end{split}
			\end{equation}
			\begin{equation}\label{twobin_Luc}
			\begin{split}
			\sum_{k = 0}^n (-1)^{(s+1)k} \binom {n}{k} \binom {m + n - k}{n} &\Big (\frac{F_{r+s}}{F_r}\Big )^k  L_{sk+p}\\
			&= \sum_{k = 0}^n (-1)^{(s+1)k}  \binom {n}{k} \binom {m}{k} \Big (\frac{F_{s}}{F_r}\Big )^k  L_{(r+s)k+p}.
			\end{split}
			\end{equation}
		
		In particular, 
		\begin{align*}
			&\sum_{k = 0}^n (-1)^{(s+1)k} \frac{ \binom {n}{k} \binom {2(2n - k)}{2n - k} \binom {2n - k}{n}}{\binom {2(n - k)}{n - k}} 
			\Big (\frac{F_{r+s}}{F_r}\Big )^k  F_{sk+p} \\
			&\qquad\qquad\qquad\qquad =  \binom {2n}{n} \sum_{k = 0}^n (-1)^{(s+1)k} \frac{\binom {n}{k}^2}{\binom {2(n - k)}{n - k}} 2^{2(n-k)}\Big (\frac{F_{s}}{F_r}\Big )^k 
			F_{(r+s)k+p},\\[4pt]
			&\sum_{k = 0}^n (-1)^{(s+1)k} \frac{ \binom {n}{k} \binom {2(2n - k)}{2n - k} \binom {2n - k}{n}}{\binom {2(n - k)}{n - k}} 
			\Big (\frac{F_{r+s}}{F_r}\Big )^k  L_{sk+p} \\
			&\qquad\qquad\qquad\qquad =  \binom {2n}{n} \sum_{k = 0}^n (-1)^{(s+1)k}\frac{\binom {n}{k}^2}{\binom {2(n - k)}{n - k}}2^{2(n-k)}  \Big (\frac{F_{s}}{F_r}\Big )^k 
			L_{(r+s)k+p}.
			\end{align*}
	\end{theorem}
	\begin{proof}
		Set $x=-F_{s}\beta^r/(F_r \alpha^{s})$ and $x=-F_{s}\alpha^r/(F_r \beta^{s})$, respectively, in \eqref{eq.iz4pi7f}, and use Lemma \ref{lem.ydalnfx}. Multiply through by $\alpha^p$ and $\beta^p$, respectively, and combine according to the Binet formulas.
	\end{proof}
	\begin{corollary}\label{cor1_twobin}
		If $n$ is a non-negative integer, $m$ is any real number and $p$ is any integer, then
		\begin{align*}
			&\sum_{k=0}^n \binom {n}{k} \binom {m + n - k}{n}  F_{p-k} = \sum_{k=0}^n \binom {n}{k} \binom {m}{k} F_{p+k},\\
			&\sum_{k=0}^n \binom {n}{k} \binom {m + n - k}{n} L_{p-k} = \sum_{k=0}^n \binom {n}{k} \binom {m}{k} L_{p+k}.
			\end{align*}
			In particular, 
		\begin{align*}
			\sum_{k = 0}^n \frac{ \binom {n}{k} \binom {2(2n - k)}{2n - k} \binom {2n - k}{n}}{\binom {2(n - k)}{n - k}}  F_{p-k} 
			&= \binom {2n}{n} \sum_{k = 0}^n \frac{\binom {n}{k}^2}{\binom {2(n - k)}{n - k}} 2^{2(n-k)} F_{p+k},
		\end{align*}
		\begin{align*}
			\sum_{k = 0}^n \frac{ \binom {n}{k} \binom {2(2n - k)}{2n - k} \binom {2n - k}{n}}{\binom {2(n - k)}{n - k}}  L_{p-k} 
			&= \binom {2n}{n} \sum_{k = 0}^n \frac{\binom {n}{k}^2}{\binom {2(n - k)}{n - k}} 2^{2(n-k)} L_{p+k}.
			\end{align*}
	\end{corollary}
	\begin{corollary}\label{cor2_twobin}
		If $n$ is a non-negative integer, $m$ is any real number and $p$ is any integer, then
			\begin{align*}
			&\sum_{k=0}^n \binom {n}{k} \binom {m + n - k}{n} F_{k+p} = \sum_{k=0}^n \binom {n}{k} \binom {m}{k} F_{2k+p},\\
			&\sum_{k=0}^n \binom {n}{k} \binom {m + n - k}{n} L_{k+p} = \sum_{k=0}^n \binom {n}{k} \binom {m}{k} L_{2k+p}.
			\end{align*}
		In particular, 
		\begin{align*}
			\sum_{k = 0}^n \frac{ \binom {n}{k} \binom {2(2n - k)}{2n - k} \binom {2n - k}{n}}{\binom {2(n - k)}{n - k}} F_{k+p} 
			&= \binom {2n}{n} \sum_{k = 0}^n \frac{\binom {n}{k}^2}{\binom {2(n - k)}{n - k}} 2^{2(n-k)} F_{2k+p},\\
			\sum_{k = 0}^n \frac{ \binom {n}{k} \binom {2(2n - k)}{2n - k} \binom {2n - k}{n}}{\binom {2(n - k)}{n - k}} L_{k+p} 
			&= \binom {2n}{n} \sum_{k = 0}^n \frac{\binom {n}{k}^2}{\binom {2(n - k)}{n - k}} 2^{2(n-k)} L_{2k+p}.
			\end{align*}
	\end{corollary}
	\begin{corollary}\label{cor3_twobin}
		If $n$ is a non-negative integer, $m$ is any real number and $p$ is any integer, then
		\begin{align}\label{eq.evj0xno}
			\sum_{k = 0}^n (\pm 2)^k  \binom {n}{k} \binom {m + n - k}{n}  F_{\left( {\frac{{3 \mp 1}}2} \right)k + p} 
			&= \sum_{k = 0}^n (\pm 1)^k \binom {n}{k} \binom {m}{k}  F_{3k + p}, \\
			\label{eq.my7undh}
			\sum_{k = 0}^n (\pm 2)^k \binom {n}{k} \binom {m + n - k}{n}  L_{\left( {\frac{{3 \mp 1}}2} \right)k + p} 
			&= \sum_{k = 0}^n (\pm 1)^k \binom {n}{k} \binom {m}{k}  L_{3k + p}.
			\end{align}
		
		In particular,
	\begin{align*}
			\sum_{k = 0}^n ( \pm 2)^k \frac{\binom {n}{k} \binom{2(2n - k)}{2n - k} \binom {2n - k}{n}}{\binom{2(n - k)}{n - k}}  
			F_{\left( {\frac{{3 \mp 1}}2} \right)k + p} &= \binom {2n}{n} \sum_{k = 0}^n (\pm 1)^k \frac{\binom nk^2}{\binom{2(n - k)}{n - k}}
			2^{2(n-k)} F_{3k + p},\\
			\sum_{k = 0}^n ( \pm 2)^k \frac{\binom {n}{k} \binom{2(2n - k)}{2n - k} \binom {2n - k}{n}}{\binom{2(n - k)}{n - k}}  
			L_{\left( {\frac{{3 \mp 1}}2} \right)k + p} & = \binom {2n}{n} \sum_{k = 0}^n (\pm 1)^k \frac{\binom nk^2}{\binom{2(n - k)}{n - k}}
			2^{2(n-k)} L_{3k + p}.
			\end{align*}
	\end{corollary}
	\begin{corollary}\label{cor4_twobin}
		If $n$ is a non-negative integer, $m$ is any real number and $p$ is any integer, then
	\begin{align*}
			\sum_{k=0}^n (-3)^k \binom {n}{k} \binom {m + n - k}{n}   F_{2k+p} &= \sum_{k=0}^n (-1)^k \binom {n}{k} \binom {m}{k}  F_{4k+p},\\
			\sum_{k=0}^n (-3)^k \binom {n}{k} \binom {m + n - k}{n}    L_{2k+p} &= 
			\sum_{k=0}^n (-1)^k \binom {n}{k} \binom {m}{k}  L_{4k+p}.
			\end{align*}
	In particular, 
	\begin{align*}
			\sum_{k = 0}^n (-3)^k  \frac{ \binom {n}{k} \binom {2(2n - k)}{2n - k} \binom {2n - k}{n}}{\binom {2(n - k)}{n - k}}  F_{2k+p} 
			= \binom {2n}{n} \sum_{k = 0}^n (-1)^k\frac{\binom {n}{k}^2}{\binom {2(n - k)}{n - k}} 2^{2(n-k)} F_{4k+p},\\
		\sum_{k = 0}^n (-3)^k \frac{ \binom {n}{k} \binom {2(2n - k)}{2n - k} \binom {2n - k}{n}}{\binom {2(n - k)}{n - k}}  L_{2k+p} 
			= \binom {2n}{n} \sum_{k = 0}^n 
			(-1)^k \frac{\binom {n}{k}^2}{\binom {2(n - k)}{n - k}} 2^{2(n-k)}  L_{4k+p}.
			\end{align*}
\end{corollary}
	\begin{theorem}
		If $n$ is a non-negative integer, $m$ is any real number and $s$, $r$ are any integers, then
	\begin{align*}
			\sum_{k=0}^n (-1)^{k(s-1)} \binom {n}{k} \binom {m + n - k}{n}  \frac{F_{r-sk}}{L_s^k} 
			&=\sum_{k=0}^n \binom {n}{k} \binom {m}{k} \frac{F_{sk+r}}{L_s^k},\\
			\sum_{k=0}^n (-1)^{k(s-1)}\binom {n}{k} \binom {m + n - k}{n} \frac{L_{r-sk}}{L_s^k} 
			&= \sum_{k=0}^n \binom {n}{k} \binom {m}{k} \frac{L_{sk+r}}{L_s^k}.
			\end{align*}
	In particular, 
	\begin{align*}
			\sum_{k = 0}^n  (-1)^{k(s-1)}\frac{\binom {n}{k} \binom {2(2n - k)}{2n - k} \binom {2n - k}{n}}{\binom {2(n - k)}{n - k}}  \frac{F_{r-sk}}{L_s^k} 
			&= \binom {2n}{n} \sum_{k = 0}^n \frac{\binom {n}{k}^2}{\binom {2(n - k)}{n - k}} 2^{2(n-k)} \frac{F_{sk+r}}{L_s^k},\\
			\sum_{k = 0}^n (-1)^{k(s-1)} \frac{ \binom {n}{k} \binom {2(2n - k)}{2n - k} \binom {2n - k}{n}}{\binom {2(n - k)}{n - k}} \frac{L_{r-sk}}{L_s^k} 
			&= \binom {2n}{n} \sum_{k = 0}^n \frac{\binom {n}{k}^2}{\binom {2(n - k)}{n - k}} 2^{2(n-k)} \frac{L_{sk+r}}{L_s^k}.
			\end{align*}
	\end{theorem}
	\begin{proof}
		Set $x=\alpha^s/L_s$ and $x=\beta^s/L_s$ in \eqref{eq.iz4pi7f}, multiply through by $\alpha^r$ and $\beta^r$, respectively, 
		and combine according to the Binet formulas. 
	\end{proof}
	\begin{theorem}
		If $n$ is a non-negative integer, $m$ is any real number, $r$ is any integer and $s$ is an odd integer, then
	\begin{align*}
			\sum_{k = 0}^n \binom {n}{k} \binom {m + n - k}{n} L_s^k F_{sk + r} &= \sum_{k = 0}^n \binom {n}{k} \binom {m}{k} F_{2sk + r},\\
			\sum_{k = 0}^n \binom {n}{k} \binom {m + n - k}{n} L_s^k L_{sk + r} &= \sum_{k = 0}^n \binom {n}{k} \binom {m}{k} L_{2sk + r}.
			\end{align*}
	In particular,
		\begin{align*}
			\sum_{k = 0}^n \frac{\binom nk\binom{2(2n - k)}{2n - k}\binom{2n - k}n}{\binom{2(n - k)}{n - k}}L_s^kF_{sk + r} 
			&= \binom{2n}n\sum_{k = 0}^n \frac{\binom nk^2}{\binom{2(n - k)}{n - k}} 2^{2(n-k)}F_{2sk + r},\\ 
			\sum_{k = 0}^n \frac{\binom nk\binom{2(2n - k)}{2n - k}\binom{2n - k}n}{\binom{2(n - k)}{n - k}}L_s^kL_{sk + r} 
			&= \binom{2n}n\sum_{k = 0}^n \frac{\binom nk^2}{\binom{2(n - k)}{n - k}}2^{2(n-k)}L_{2sk + r}. 
			\end{align*}
	\end{theorem}
	\begin{proof}
		Set $x=\alpha^{2s}$ in \eqref{eq.iz4pi7f} and use the fact that $\alpha^{2s} - 1=\alpha^sL_s$ if $s$ is an odd integer.
	\end{proof}
	\begin{theorem}
		Let $n$ be a non-negative integer and let $m$ be a real number. Let $r$, $s$, $p$ and $t$ be any integers. Then
	\begin{equation}\label{eq.v6kzdd7}
			\begin{split}
			&\sum_{k = 0}^n {( - 1)^k \binom nk\binom{m + n - k}{n} \frac{L_{r + s}^k}{L_r^k} L_{s(p - k) + t} } = \sum_{k = 0}^{\left\lfloor {n/2} \right\rfloor } {\binom n{2k}\binom m{2k} \frac{5^k F_s^{2k}}{ L_r^{2k}} L_{sp - 2k(r + s) + t} } \\
			&\qquad\qquad + ( - 1)^r \sum_{k = 1}^{\left\lceil {n/2} \right\rceil } {\binom n{2k - 1}\binom m{2k - 1} \frac{5^k F_s^{2k - 1}}{L_r^{2k - 1}} F_{sp - (2k - 1)(r + s) + t} } ,
			\end{split}
			\end{equation}
		\begin{equation}\label{eq.l1tpqld}
			\begin{split}
			&\sum_{k = 0}^n {( - 1)^k \binom nk\binom{m + n - k}n \frac{L_{r + s}^k}{ L_r^{k}} F_{s(p - k) + t} } = \sum_{k = 0}^{\left\lfloor {n/2} \right\rfloor } {\binom n{2k}\binom m{2k}  \frac{5^k F_s^{2k}} {L_r^{2k}} F_{sp - 2k(r + s) + t} } \\
			&\qquad\qquad + ( - 1)^r \sum_{k = 1}^{\left\lceil {n/2} \right\rceil } {\binom n{2k - 1}\binom m{2k - 1} \frac{5^{k - 1} F_s^{2k - 1}} {L_r^{2k - 1}} L_{sp - (2k - 1)(r + s) + t} } .
			\end{split}
			\end{equation}
	\end{theorem}
	\begin{proof}
		Set $x = \sqrt 5 \beta ^r F_s  /(L_r \alpha ^s) $ in \eqref{eq.iz4pi7f} and use \eqref{es.benssj4} to obtain
		\begin{align*}
			\sum_{k = 0}^n ( - 1)^k &\binom nk\binom{m + n - k}nL_r^{p - k} L_{r + s}^k \alpha^{s(p - k) + t}  \\
			&\qquad\qquad= \sum_{k = 0}^n {( - 1)^{kr} \binom nk\binom mk (\sqrt 5)^k L_r^{p - k} F_s^k \alpha^{sp - (r + s)k + t}},
			\end{align*}
	from which the results follow.
	\end{proof}
	\begin{corollary}
		Let $n$ be a non-negative integer; let $m$ be a real number and let $s$ and $p$ be any integers. Then
		\begin{align*}
			&\sum_{k = 0}^n ( - 1)^k \binom nk\binom{m + n - k}{n}  \frac{ 2^{k - 1}L_{sk}}{L_s^k}   = \sum_{k = 0}^{\left\lfloor {n/2} \right\rfloor } \binom n{2k}\binom m{2k} \frac{5^kF_s^{2k}}{L_s^{2k}},\\
			&\sum_{k = 0}^n ( - 1)^{k} \binom nk\binom{m + n - k}n  \frac{2^{k - 1}F_{sk}}{L_s^{ k}}  = - \sum_{k = 1}^{\left\lceil {n/2} \right\rceil } \binom n{2k - 1}\binom m{2k - 1} \frac{ 5^{k - 1}F_s^{2k - 1}}{L_s^{ 2k - 1}}.
			\end{align*}
	\end{corollary}
	\begin{proof}
		Set $t=-sp$ and $r=-s$ in \eqref{eq.v6kzdd7} and \eqref{eq.l1tpqld}.
	\end{proof}
\begin{corollary}
		Let $n$ be a non-negative integer and let $s$ and $p$ be any integers. Then
	\begin{align*}
			&\sum_{k = 0}^n ( - 2)^k \frac{\binom nk\binom{2(2n - k)}{2n - k}\binom{2n - k}n}{\binom{2(n - k)}{n - k}}\frac{L_{sk}}{L_s^{ k}}   = \binom{2n}n\sum_{k = 0}^{\left\lfloor {n/2} \right\rfloor } {\frac{\binom n{2k}^2}{\binom{2(n - 2k)}{n - 2k}}  \frac{5^k 2^{2(n-2k)+1} F_s^{2k}}{L_s^{2k}} },\\
			&\sum_{k = 0}^n ( - 2)^{k} \frac{\binom nk\binom{2(2n - k)}{2n - k}\binom{2n - k}n}{\binom{2(n - k)}{n - k}}  \frac{F_{sk}}{L_s^{ k}} = - \binom{2n}n\sum_{k = 1}^{\left\lceil {n/2} \right\rceil } \frac{\binom n{2k - 1}^2}{\binom{2(n - 2k + 1)}{n - 2k + 1}} \frac{2^{2(n-2k+1)+1}5^{k - 1} F_s^{2k - 1}}{L_s^{2k - 1}}.
			\end{align*}
	\end{corollary}

	\section{Still other classes of identities with two binomial coefficients}
	
	\begin{lemma}[{\cite[Identity (3.18)]{Gould2}}]
		If $n$ is a non-negative integer, $m$ is any real number and $x$, $y$ are any complex variables, then
	\begin{equation}\label{eq.jp8xzwd}
			\sum_{k=0}^n{\binom nk\binom{m + k}k{(x - y)}^{n - k}y^k}=\sum_{k=0}^n{\binom nk\binom mk x^{n - k}y^k}.
			\end{equation}
\end{lemma}
	\begin{theorem}
		If $n$ is a non-negative integer, $m$ is any real number and $r$, $s$ and $t$ are any integers, then
	\begin{equation*}
			\sum_{k = 0}^n {\binom nk\binom{m + k}kq^{rk} U_r^{n - k} U_s^kW_{t - (r + s)k + sn} }  = \sum_{k = 0}^n {\binom nk\binom mkq^{rk} U_{r + s}^{n - k} U_s^kW_{t - rk} }.
			\end{equation*}
\end{theorem}
	\begin{proof}
		Set $x=U_{r + s}$ and $y=\tau^rU_s$ in \eqref{eq.jp8xzwd} and multiply by $\sigma^t$ to obtain
		\begin{equation}\label{eq.buukm6j}
			\sum_{k = 0}^n {\binom nk\binom{m + k}k\sigma ^{s(n - k) + t} \tau ^{rk} U_r^{n - k} U_s^k }  = \sum_{k = 0}^n {\binom nk\binom mk \sigma ^t \tau ^{rk} U_{r + s}^{n - k} U_s^k }.
			\end{equation}
	Similarly obtain
		\begin{equation}\label{eq.clwld8j}
			\sum_{k = 0}^n {\binom nk\binom{m + k}k\sigma ^{rk} \tau ^{s(n - k) + t}  U_r^{n - k} U_s^k }  = \sum_{k = 0}^n {\binom nk\binom mk \sigma ^{rk} \tau ^t U_{r + s}^{n - k} U_s^k }.
			\end{equation}
		
		Combine \eqref{eq.buukm6j} and \eqref{eq.clwld8j} using the Binet formula.
	\end{proof}
	In particular,
		\begin{equation*}
		\sum_{k = 0}^n {\binom nk\binom{m + k}kq^{rk} U_r^{n - k} U_s^kV_{t - (r + s)k + sn} }  = \sum_{k = 0}^n {\binom nk\binom mkq^{rk} U_{r + s}^{n - k} U_s^kV_{t - rk} }
		\end{equation*}
	and
		\begin{equation*}
		\sum_{k = 0}^n {\binom nk\binom{m + k}kq^{rk} U_r^{n - k} U_s^kU_{t - (r + s)k + sn} }  = \sum_{k = 0}^n {\binom nk\binom mkq^{rk} U_{r + s}^{n - k} U_s^kU_{t - rk} };
		\end{equation*}
	with the special cases
		\begin{align*}
		&\sum_{k = 0}^n (-1)^{rk} \binom nk\binom{m + k}k F_r^{n - k} F_s^kL_{t - (r + s)k + sn} = \sum_{k = 0}^n (-1)^{rk} \binom nk\binom mk F_{r + s}^{n - k} F_s^kL_{t - rk} ,\\
		&\sum_{k = 0}^n (-1)^{rk} \binom nk\binom{m + k}k F_r^{n - k} F_s^kF_{t - (r + s)k + sn}  = \sum_{k = 0}^n (-1)^{rk} \binom nk\binom mk F_{r + s}^{n - k} F_s^kF_{t - rk} .
		\end{align*}
		\begin{lemma}[{\cite[Identity (3.84)]{Gould2}}]
		If $n$ is a non-negative integer and $x$ is any complex variable, then
		\begin{equation}\label{eq.vft8g6w}
			\sum_{k = 0}^n {( - 1)^k \binom nk\binom{2k}k2^{2n - 2k} x^k }  = \sum_{k = 0}^n {( - 1)^k \binom{2(n - k)}{n - k}\binom{2k}k(x - 1)^k }.
			\end{equation}
	\end{lemma}
	
\begin{theorem}
If $n$ is a non-negative integer and $r$, $s$, $t$ and $m$ are any integers, then
\begin{equation*}
\begin{split}
\sum_{k = 0}^n & {( - 1)^k \binom nk\binom{2k}k2^{2(n - k)} U_s^{m - k} U_{r + s}^k W_{t + r(m - k)} } \\
&\qquad = \sum_{k = 0}^n {( - 1)^k \binom{2(n - k)}{n - k}\binom{2k}kq^{sk} U_s^{m - k} U_r^k W_{t + r(m - k) - sk} } .
\end{split}
\end{equation*}
\end{theorem}	
\begin{proof}
Set $x=U_{r + s}/(\tau^rU_s)$ in \eqref{eq.vft8g6w}; use~\eqref{eq.iamiky1} and multiply through the resulting equation by $\tau^t$ to obtain
\begin{equation}\label{eq.bt9i0tt}
\begin{split}
\sum_{k = 0}^n & ( - 1)^k  \binom nk\binom{2k}k2^{2(n - k)} U_s^{m - k} U_{r + s}^k \tau^{t + r(m - k)}  \\
&\qquad = \sum_{k = 0}^n {( - 1)^k \binom{2(n - k)}{n - k}\binom{2k}k U_s^{m - k} U_r^k \tau^{t + r(m - k)}\sigma^{sk} } .
\end{split}
\end{equation}
Similarly obtain
\begin{equation}\label{eq.hjg15mn}
\begin{split}
\sum_{k = 0}^n & {( - 1)^k \binom nk\binom{2k}k2^{2(n - k)} U_s^{m - k} U_{r + s}^k \sigma^{t + r(m - k)} } \\
&\qquad = \sum_{k = 0}^n {( - 1)^k \binom{2(n - k)}{n - k}\binom{2k}k U_s^{m - k} U_r^k \sigma^{t + r(m - k)}\tau^{sk} } .
\end{split}
\end{equation}
Combine \eqref{eq.bt9i0tt} and \eqref{eq.hjg15mn} according to the Binet formula.
\end{proof}
In particular,
\begin{align*}
\sum_{k = 0}^n &{( - 1)^k \binom nk\binom{2k}k2^{2(n - k)} U_s^{m - k} U_{r + s}^k U_{t + r(m - k)} } \\
&\qquad = \sum_{k = 0}^n {( - 1)^k \binom{2(n - k)}{n - k}\binom{2k}kq^{sk} U_s^{m - k} U_r^k U_{t + r(m - k) - sk} }
\end{align*}
and
\begin{align*}
\sum_{k = 0}^n &{( - 1)^k \binom nk\binom{2k}k2^{2(n - k)} U_s^{m - k} U_{r + s}^k V_{t + r(m - k)} } \\
&\qquad = \sum_{k = 0}^n {( - 1)^k \binom{2(n - k)}{n - k}\binom{2k}kq^{sk} U_s^{m - k} U_r^k V_{t + r(m - k) - sk} } ;
\end{align*}
with the special cases
\begin{align*}
\sum_{k = 0}^n ( - 1)^k \binom nk&\binom{2k}k 2^{2(n - k)} F_s^{m - k} F_{r + s}^k F_{t + r(m - k)} \\
&\qquad = \sum_{k = 0}^n {( - 1)^k \binom{2(n - k)}{n - k}\binom{2k}kq^{sk} F_s^{m - k} F_r^k F_{t + r(m - k) - sk} },\\
\sum_{k = 0}^n ( - 1)^k \binom nk&\binom{2k}k 2^{2(n - k)} F_s^{m - k} F_{r + s}^k L_{t + r(m - k)} \\
&\qquad = \sum_{k = 0}^n {( - 1)^k \binom{2(n - k)}{n - k}\binom{2k}kq^{sk} F_s^{m - k} F_r^k L_{t + r(m - k) - sk} }.
\end{align*}
\begin{lemma}[{\cite{Gould3,Simon01}}]
		If $n$ is a non-negative integer and $x$ is any complex variable, then
		\begin{equation}\label{eq.bgk0pz5}
			\sum_{k = 0}^n ( - 1)^{n - k} \binom nk\binom{n + k}k(1 + x)^k   = \sum_{k = 0}^n \binom nk\binom{n + k}kx^k .
			\end{equation}
	\end{lemma}
\begin{theorem}
If $n$ is a non-negative integer and $s$, $t$ and $m$ are any integers, then
\begin{align}
&\sum_{k = 0}^n {( - 1)^{n - k} \binom nk\binom{n + k}kW_{2s(m - k) + t} }  = \sum_{k = 0}^n {\binom nk\binom{n + k}kq^{sk} W_{2s(m - k) + t} },\label{eq.tbzd537}\\[4pt]
&\sum_{k = 0}^n ( - 1)^{n - k} \binom nk\binom{n + k}k \frac{V_s^k W_{sk + t} }{q^{sk}}  = \sum_{k = 0}^n {\binom nk\binom{n + k}k\frac{{W_{2sk + t} }}{{q^{sk} }}} \label{eq.gkblhan}.
\end{align}
\end{theorem}
\begin{proof}
To prove~\eqref{eq.tbzd537}, set $x=q^s/\tau^{2s}$ and $x=q^s/\sigma^s$, in turn, in \eqref{eq.bgk0pz5} and use~\eqref{eq.u9uuagc}. Combine the resulting equations according to the Binet formula. For \eqref{eq.gkblhan}, use $x=\tau^{2s}/q^s$ and $x=\sigma^s/q^s$, in turn, in \eqref{eq.bgk0pz5}.
\end{proof}
In particular,
\begin{align*}
&\sum_{k = 0}^n {( - 1)^{n - k} \binom nk\binom{n + k}kU_{2s(m - k) + t} }  = \sum_{k = 0}^n {\binom nk\binom{n + k}kq^{sk} U_{2s(m - k) + t} },\\[4pt]
&\sum_{k = 0}^n {( - 1)^{n - k} \binom nk\binom{n + k}k\frac{{V_s^k U_{sk + t} }}{{q^{sk} }}}  = \sum_{k = 0}^n {\binom nk\binom{n + k}k\frac{{U_{2sk + t} }}{{q^{sk} }}} ,
\end{align*}
and
\begin{align*}
&\sum_{k = 0}^n {( - 1)^{n - k} \binom nk\binom{n + k}kV_{2s(m - k) + t} }  = \sum_{k = 0}^n {\binom nk\binom{n + k}kq^{sk} V_{2s(m - k) + t} },\\[4pt] 
&\sum_{k = 0}^n {( - 1)^{n - k} \binom nk\binom{n + k}k \frac{{V_s^k V_{sk + t} }}{q^{sk} }}  = \sum_{k = 0}^n {\binom nk\binom{n + k}k\frac{{V_{2sk + t} }}{{q^{sk} }}} ;
\end{align*}
and the special cases
\begin{align*}
&\sum_{k = 0}^n {( - 1)^{n - k} \binom nk\binom{n + k}kF_{2s(m - k) + t} }  = \sum_{k = 0}^n (-1)^{sk} {\binom nk\binom{n + k}k F_{2s(m - k) + t} },\\ 
&\sum_{k = 0}^n {( - 1)^{n + (s - 1)k} \binom nk\binom{n + k}kL_s^k F_{sk + t} }  = \sum_{k = 0}^n (-1)^{sk} \binom nk\binom{n + k}k F_{2sk + t} 
\end{align*}
and
\begin{align*}
&\sum_{k = 0}^n {( - 1)^{n - k} \binom nk\binom{n + k}kL_{2s(m - k) + t} }  = \sum_{k = 0}^n (-1)^{sk} {\binom nk\binom{n + k}k  L_{2s(m - k) + t} },\\ 
&\sum_{k = 0}^n {( - 1)^{n + (s - 1)k} \binom nk\binom{n + k}k L_s^k L_{sk + t} }  = \sum_{k = 0}^n (-1)^{sk}  {\binom nk\binom{n + k}kL_{2sk + t} } .
\end{align*}

	Next we present an obvious extension of \eqref{eq.g85la9i} and some associated Fibonacci--Lucas sums.
	\begin{lemma}
		Let $x$ and $y$ be complex variables. Let $m$ and $n$ be non-negative integers and let $r$ be any integer. Then
	\begin{equation}\label{eq.fm7cy8m}
			\sum_{k = 0}^n (-1)^k{\binom{m - n + k}k\binom{n - k}rx^{n - k - r} y^k }  = \sum_{k = 0}^n (-1)^k \binom{m + 1}k\binom{n - k}r(x - y)^{n - k - r} y^k.
			\end{equation}
\end{lemma}
\begin{theorem}
		If $m$ and $n$ are non-negative integers and $s$, $r$ and $t$ are any integers, then
	\begin{align*}
			\sum_{k = 0}^n &( - 1)^k \binom{m - n + k}k\binom{n - k}rq^{s(n - k - r)} W_{2sk + t} \\
			&\qquad\qquad\qquad  = \sum_{k = 0}^n {( - 1)^k \binom{m + 1}k\binom{n - k}rV_s^{n - k - r} W_{s(n + k - r) + t} }.
			\end{align*}
\end{theorem}
	\begin{proof}
		Choose $x=q^s$ and $y=\tau^{2s}$ in \eqref{eq.fm7cy8m}, use \eqref{eq.u9uuagc} and multiply through by $\tau^t$ to obtain
\begin{equation}\label{eq.kxm5i5k}
			\begin{split}
			\sum_{k = 0}^n ( - 1)^k& \binom{m - n + k}k\binom{n - k}rq^{s(n - k - r)}\tau^{2sk + t} \\
			&\qquad \qquad= \sum_{k = 0}^n ( - 1)^k \binom{m + 1}k\binom{n - k}r \tau^{s(n + k - r) + t} V_s^{n - k - r}  .
			\end{split}
			\end{equation}
		Similarly obtain
	\begin{equation}\label{eq.jcb7kdd}
			\begin{split}
			\sum_{k = 0}^n ( - 1)^k &\binom{m - n + k}k\binom{n - k}rq^{s(n - k - r)}\sigma^{2sk + t}  \\
			&\qquad\qquad = \sum_{k = 0}^n ( - 1)^k \binom{m + 1}k\binom{n - k}r \sigma^{s(n + k - r) + t} 
			V_s^{n - k - r},
			\end{split}
			\end{equation}
		
		The result follows from \eqref{eq.kxm5i5k}, \eqref{eq.jcb7kdd} and the Binet formula.
	\end{proof}
	
	In particular,
\begin{align*}
		&\sum_{k = 0}^n {( - 1)^k \binom{m - n + k}k\binom{n - k}rq^{s(n - k - r)} U_{2sk + t} }\\
		&\qquad\qquad\qquad  = \sum_{k = 0}^n {( - 1)^k \binom{m + 1}k\binom{n - k}rV_s^{n - k - r} U_{s(n + k - r) + t} },\\
		&\sum_{k = 0}^n {( - 1)^k \binom{m - n + k}k\binom{n - k}rq^{s(n - k - r)} V_{2sk + t} }\\
		&\qquad\qquad\qquad  = \sum_{k = 0}^n {( - 1)^k \binom{m + 1}k\binom{n - k}rV_s^{n - k - r} V_{s(n + k - r) + t} };
		\end{align*}
with the special cases
	\begin{align*}
		&\sum_{k = 0}^n ( - 1)^{k(s + 1)} \binom{m - n + k}k\binom{n - k}rL_{2sk + t} \\
		&\qquad\qquad = \frac{( - 1)^{s(n - r)}}{L_s^r} \sum_{k = 0}^n {( - 1)^k \binom{m + 1}k\binom{n - k}r L_s^{n - k} L_{s(n + k - r) + t} },\\
		&\sum_{k = 0}^n ( - 1)^{k(s + 1)} \binom{m - n + k}k\binom{n - k}r F_{2sk + t} \\
		&\qquad\qquad = \frac{( - 1)^{s(n - r)}}{L_s^r} \sum_{k = 0}^n {( - 1)^k \binom{m + 1}k \binom{n - k}rL_s^{n - k} F_{s(n + k - r) + t} }.
	\end{align*}
\begin{theorem}
		Let $m$, $n$, $r$, $s$ and $t$ be integers with $n\ge0$. If $n$ and $r$ have the same parity then
			\begin{equation*}
			\begin{split}
			&W_t \sum_{k = 0}^{\left\lfloor {n/2} \right\rfloor } {\binom{m - n + 2k}{2k}\binom{n - 2k}r \Delta^{n - 2k - r} U_s^{n - 2k - r} V_s^{2k}  }\\
			&\qquad  - \big(W_{t + 1} - qW_{t - 1}\big) \sum_{k = 1}^{\left\lceil {n/2} \right\rceil } {\binom{m - n + 2k - 1}{2k - 1}\binom{n - 2k + 1}r \Delta^{n - 2k - r} U_s^{n - 2k  - r+ 1} V_s^{2k - 1}  }\\ 
			&\qquad\qquad = \sum_{k = 0}^n {( - 1)^k \binom{m + 1}k\binom{n - k}r2^{n - k - r} V_s^k W_{s(n - k - r) + t} } ,
			\end{split}
			\end{equation*}
		while if $n$ and $r$ have different parities, then
	\begin{equation*}
			\begin{split}
			&\big(W_{t + 1} - qW_{t - 1}\big) \sum_{k = 0}^{\left\lfloor {n/2} \right\rfloor } \binom{m - n + 2k}{2k}\binom{n - 2k}r \Delta^{n - 2k - r - 1} U_s^{n - 2k - r} V_s^{2k}  \\
			&\qquad  - W_t \sum_{k = 1}^{\left\lceil {n/2} \right\rceil } {\binom{m - n + 2k - 1}{2k - 1}\binom{n - 2k + 1}r \Delta^{n - 2k  + 1- r} U_s^{n - 2k  - r+ 1 } V_s^{2k - 1}  }\\ 
			&\qquad\qquad = \sum_{k = 0}^n {( - 1)^k \binom{m + 1}k\binom{n - k}r2^{n - k - r} V_s^k W_{s(n - k - r) + t} }.
			\end{split}
			\end{equation*}
\end{theorem}
	\begin{proof}
		Set $y=V_s$ and $x=\Delta U_s$ in \eqref{eq.fm7cy8m}, use Lemma \ref{lem.w65xm59} and the summation identity
	\begin{equation*}
			\sum_{j = 0}^n {f_j }  = \sum_{j = 0}^{\left\lfloor {n/2} \right\rfloor } {f_{2j} }  + \sum_{j = 1}^{\left\lceil {n/2} \right\rceil } {f_{2j - 1} }.
			\end{equation*}
	\end{proof}
	In particular, if $n$ and $r$ have the same parity then
	\begin{align}
		&U_t \sum_{k = 0}^{\left\lfloor {n/2} \right\rfloor } \binom{m - n + 2k}{2k}\binom{n - 2k}rU_s^{n - 2k - r} V_s^{2k} \Delta^{n - 2k - r} \notag\\
		&\qquad\qquad  - V_t \sum_{k = 1}^{\left\lceil {n/2} \right\rceil } \binom{m - n + 2k - 1}{2k - 1}\binom{n - 2k + 1}r \Delta^{n - 2k - r}  U_s^{n - 2k  - r+ 1} V_s^{2k - 1} \label{eq.i2se1iz}\\ 
		&\qquad\qquad\qquad\qquad = \sum_{k = 0}^n {( - 1)^k \binom{m + 1}k\binom{n - k}r2^{n - k - r} V_s^k U_{s(n - k - r) + t} }\notag
		\end{align}
and
	\begin{align*}
		&V_t \sum_{k = 0}^{\left\lfloor {n/2} \right\rfloor } \binom{m - n + 2k}{2k} \binom{n - 2k}r \Delta^{n - 2k - r}  U_s^{n - 2k - r} V_s^{2k}\notag \\
		&\qquad\qquad  - U_t \sum_{k = 1}^{\left\lceil {n/2} \right\rceil } \binom{m - n + 2k - 1}{2k - 1}\binom{n - 2k + 1}r \Delta^{n - 2k - r+2} U_s^{n - 2k  - r+ 1} V_s^{2k - 1}\\ 
		&\qquad\qquad\qquad\qquad = \sum_{k = 0}^n {( - 1)^k \binom{m + 1}k\binom{n - k}r2^{n - k - r} V_s^k V_{s(n - k - r) + t} } ; \notag
		\end{align*}
while if $n$ and $r$ have different parities, then
	\begin{align*}
		&V_t \sum_{k = 0}^{\left\lfloor {n/2} \right\rfloor } \binom{m - n + 2k}{2k}\binom{n - 2k}r \Delta^{n - 2k - r - 1} U_s^{n - 2k - r}  V_s^{2k} \notag\\
		&\qquad\qquad  - U_t \sum_{k = 1}^{\left\lceil {n/2} \right\rceil } \binom{m - n + 2k - 1}{2k - 1}\binom{n - 2k + 1}r \Delta^{n - 2k - r + 1} U_s^{n - 2k  - r + 1} V_s^{2k - 1}  \\ 
		&\qquad\qquad\qquad\qquad = \sum_{k = 0}^n {( - 1)^k \binom{m + 1}k\binom{n - k}r2^{n - k - r} V_s^k U_{s(n - k - r) + t} }\notag,
		\end{align*}
and
	\begin{equation}\label{eq.q375501}
		\begin{split}
		&U_t \sum_{k = 0}^{\left\lfloor {n/2} \right\rfloor } \binom{m - n + 2k}{2k}\binom{n - 2k}r \Delta^{n - 2k - r + 1} U_s^{n - 2k - r} V_s^{2k} \\
		&\qquad\qquad  - V_t \sum_{k = 1}^{\left\lceil {n/2} \right\rceil } \binom{m - n + 2k - 1}{2k - 1}\binom{n - 2k + 1}r \Delta^{n - 2k - r + 1} U_s^{n - 2k  - r+ 1} V_s^{2k - 1} \\ 
		&\qquad\qquad\qquad\qquad = \sum_{k = 0}^n {( - 1)^k \binom{m + 1}k\binom{n - k}r2^{n - k - r} V_s^k V_{s(n - k - r) + t} } ;
		\end{split}
		\end{equation}
with the special cases:
	\begin{align*}
		&L_t \sum_{k = 0}^{\left\lfloor {n/2} \right\rfloor } \binom{m - n + 2k}{2k}\binom{n - 2k}r 5^{(n - r)/2 - k} F_s^{n - 2k - r} L_s^{2k}  \\
		&\qquad\qquad  - F_t \sum_{k = 1}^{\left\lceil {n/2} \right\rceil } \binom{m - n + 2k - 1}{2k - 1}\binom{n - 2k + 1}r 5^{(n - r)/2 - k + 1} F_s^{n - 2k  - r+ 1} L_s^{2k - 1} \\ 
		&\qquad\qquad\qquad\qquad = \sum_{k = 0}^n {( - 1)^k \binom{m + 1}k\binom{n - k}r2^{n - k - r} L_s^k L_{s(n - k - r) + t} },\\[5pt]
		&F_t \sum_{k = 0}^{\left\lfloor {n/2} \right\rfloor } \binom{m - n + 2k}{2k}\binom{n - 2k}r 5^{(n - r)/2 - k} F_s^{n - 2k - r} L_s^{2k}  \\
		&\qquad\qquad  - L_t \sum_{k = 1}^{\left\lceil {n/2} \right\rceil } \binom{m - n + 2k - 1}{2k - 1}\binom{n - 2k + 1}r 5^{(n - r)/2 - k} F_s^{n - 2k  - r+ 1} L_s^{2k - 1} \\ 
		&\qquad\qquad\qquad\qquad = \sum_{k = 0}^n {( - 1)^k \binom{m + 1}k\binom{n - k}r2^{n - k - r} L_s^k F_{s(n - k - r) + t} }; 
		\end{align*}
while if $n$ and $r$ have different parities then
		\begin{align*}
		&F_t \sum_{k = 0}^{\left\lfloor {n/2} \right\rfloor } \binom{m - n + 2k}{2k}\binom{n - 2k}r 5^{(n - r + 1)/2 - k} F_s^{n - 2k - r} L_s^{2k}\\
		&\qquad\qquad  - L_t \sum_{k = 1}^{\left\lceil {n/2} \right\rceil } \binom{m - n + 2k - 1}{2k - 1}\binom{n - 2k + 1}r 5^{(n - r + 1)/2 - k} F_s^{n - 2k  - r+ 1} L_s^{2k - 1} \\ 
		&\qquad\qquad\qquad\qquad = \sum_{k = 0}^n {( - 1)^k \binom{m + 1}k\binom{n - k}r 2^{n - k - r} L_s^k L_{s(n - k - r) + t} },\\[5pt]
		&L_t \sum_{k = 0}^{\left\lfloor {n/2} \right\rfloor } \binom{m - n + 2k}{2k}\binom{n - 2k}r 5^{(n - r + 1)/2 - k} F_s^{n - 2k - r} L_s^{2k} \\
		&\qquad\qquad  - F_t \sum_{k = 1}^{\left\lceil {n/2} \right\rceil } \binom{m - n + 2k - 1}{2k - 1}\binom{n - 2k + 1}r 5^{(n - r - 1)/2 - k} F_s^{n - 2k - r+ 1 } L_s^{2k - 1}  \\ 
		&\qquad\qquad\qquad\qquad = \sum_{k = 0}^n ( - 1)^k \binom{m + 1}k\binom{n - k}r2^{n - k - r} L_s^k F_{s(n - k - r) + t}.
		\end{align*}
	
Note that in simplifying \eqref{eq.i2se1iz}--\eqref{eq.q375501}, we used \eqref{eq.zqa5sbx}.
	\begin{corollary}
		If $m$, $n$, $r$ and $s$ are integers, then
			\begin{align*}
			&\sum\limits_{k = 0}^n {( - 1)^k \binom{m + 1}k\binom{n - k}r2^{n - k - r - 1} L_s^k L_{s(n - k - r)} }\\
			&\qquad=\begin{cases}
			\sum\limits_{k = 0}^{\left\lfloor {n/2} \right\rfloor } \binom{m - n + 2k}{2k}\binom{n - 2k}r 5^{(n - r)/2 - k} F_s^{n - 2k - r} L_s^{2k} , & \text{\rm if $n - r$ is even;}\\
			\\
			-\sum\limits_{k = 1}^{\left\lceil {n/2} \right\rceil } \binom{m - n + 2k - 1}{2k - 1}\binom{n - 2k + 1}r 5^{(n - r + 1)/2 - k} F_s^{n - 2k  - r+ 1} L_s^{2k - 1}, & \text{\rm otherwise;}
			\end{cases}\\[5pt]
			&\sum\limits_{k = 0}^n {( - 1)^k \binom{m + 1}k\binom{n - k}r2^{n - k - r - 1} L_s^k F_{s(n - k - r)} }\\
			&\qquad=\begin{cases}
			-\sum\limits_{k = 1}^{\left\lceil {n/2} \right\rceil } \binom{m - n + 2k - 1}{2k - 1}\binom{n - 2k + 1}r 5^{(n - r)/2 - k}  F_s^{n - 2k - r+ 1 } L_s^{2k - 1} , & \text{\rm if $n - r$ is even;}\\
			\\
			\sum\limits_{k = 0}^{\left\lfloor {n/2} \right\rfloor }\binom{m - n + 2k}{2k}\binom{n - 2k}r 5^{(n - r + 1)/2 - k} F_s^{n - 2k - r} L_s^{2k} , & \text{\rm otherwise.}
			\end{cases}
			\end{align*}
	\end{corollary}

	\section{Identities with three binomial coefficients}

	Concerning identities with three binomial coefficients some classical Fibonacci (Lucas) examples exist. 
	For instance, Carlitz \cite{Carlitz_AP3} presented the identities
	\begin{equation*}
		\sum_{k = 0}^n \binom {n}{k}^3 F_k = \sum_{2k\leq n} \frac{(n+k)!}{(k!)^3 (n-2k)!} F_{2n-3k}
		\end{equation*}
and
	\begin{equation*}
		\sum_{k = 0}^n \binom {n}{k}^3 L_k = \sum_{2k\leq n} \frac{(n+k)!}{(k!)^3 (n-2k)!} L_{2n-3k}.
		\end{equation*}
	
	In addition, Zeitlin \cite{Zeitlin_AP1,Zeitlin_AP3} derived
	\begin{equation*}
		\sum_{k = 0}^{2n} \binom {2n}{k}^3 F_{2k} = F_{2n} \sum_{k=0}^n \frac{(2n+k)!}{(k!)^3 (2n-2k)!} 5^{n-k},
		\end{equation*}
		\begin{equation*}
		\sum_{k = 0}^{2n} \binom {2n}{k}^3 L_{2k} = L_{2n} \sum_{k=0}^n \frac{(2n+k)!}{(k!)^3 (2n-2k)!} 5^{n-k}.
		\end{equation*}
	
In his solution to Carlitz' proposal from above Zeitlin \cite{Zeitlin_AP2} proved {\it mutatis mutandis} the identity
	\begin{equation*}
		\sum_{k = 0}^n \binom {n}{k}^3 (-q)^{n-k} p^k W_k = \sum_{2k\leq n} \frac{(n+k)!}{(k!)^3 (n-2k)!} p^k (-q)^k W_{2n-3k}.
		\end{equation*}

	His results are based on the polynomial identity	
		\begin{equation*}
		\sum_{k = 0}^n \binom {n}{k}^3 x^k = \sum_{2k\leq n} \frac{(n+k)!}{(k!)^3 (n-2k)!} x^k (x+1)^{n-2k}.
		\end{equation*}
	
In this section we provide more examples of this kind using ``Zeitlin's identity'' in its equivalent form given in the next lemma.
	\begin{lemma}[{\cite[Identity (6.7)]{Gould2}},\cite{Sun12}]
		If $n$ is a non-negative integer and $x$ is any complex variable, then
	\begin{equation}\label{eq.tacrojp}
			\sum_{k = 0}^n {\binom nk^3 x^k }  = \sum_{k = 0}^n {\binom{n + k}{2k}\binom{2k}k\binom{n - k}kx^k(1 + x)^{n - 2k} }.
			\end{equation}
	\end{lemma}
	\begin{theorem}
		If $n$ is a non-negative integer and $r$ and $t$ are any integers, then
	\begin{equation*}
			\sum_{k = 0}^n {\binom nk^3 q^{rk} W_{r(n - 2k) + t} }  = W_t \sum_{k = 0}^n {\binom{n + k}{2k}\binom{2k}k\binom{n - k}kq^{rk} V_r^{n - 2k} }.
			\end{equation*}
\end{theorem}
	\begin{proof}
		Set $x=\tau^r/\sigma^r$ and $x=\sigma^r/\tau^r$, in turn, in \eqref{eq.tacrojp}. Combine according to the Binet formula.
	\end{proof}
	\begin{corollary}
		If $n$ is a non-negative integer and $r$ is any integer, then
	\begin{equation*}
			\sum_{k = 0}^n {\binom nk^3 q^{rk} U_{r(n - 2k)} }  = 0.
			\end{equation*}
	In particular,
	\begin{equation*}
			\sum_{k = 0}^n {\binom nk^3 (-1)^{rk} F_{r(n - 2k)} }  = 0.
			\end{equation*}
\end{corollary}
	\begin{theorem}
		If $n$ is a non-negative integer and $r$, $s$ and $t$ are any integers, then
	\begin{align*}
			&\sum_{k = 0}^n ( - 1)^k \binom nk^3  U_{r + s}^{n - k} U_s^k W_{t + rk}  \\
			&\qquad = \sum_{k = 0}^n ( - 1)^k \binom{n + k}{2k}\binom{2k}k\binom{n - k}k q^{s(n - 2k)} U_{r + s}^k U_s^k U_r^{n - 2k}  W_{t + rk - s(n - 2k)}.
			\end{align*}
\end{theorem}
	\begin{proof}
		Set $x=-\tau^rU_s/U_{r + s}$ and $x=-\sigma^rU_s/U_{r + s}$ in \eqref{eq.tacrojp}, in turn, bearing in mind \eqref{eq.iamiky1} and \eqref{eq.zy0gfyn}. Combine the resulting equations using the Binet formula and Lemma \ref{lem.w65xm59}.
	\end{proof}
\begin{corollary}
		If $n$ is a non-negative integer and $r$ and $t$ are any integers, then
	\begin{equation*}
			\sum_{k = 0}^n ( - 1)^k \binom nk^3  V_r^{n - k} W_{t + rk}  = \sum_{k = 0}^n ( - 1)^k \binom{n + k}{2k}\binom{2k}k\binom{n - k}k q^{r(n - 2k)} V_r^k  W_{t + r(3k - n)}.
			\end{equation*}
\end{corollary}

	\section{Concluding comments}

		Further identities with two binomial coefficients can be derived from Lemma \ref{lem.rhz91sw} 
	which is a generalization of \eqref{eq.wl5p3v8}.
	\begin{lemma}\label{lem.rhz91sw}
		Let $x$ and $y$ be complex variables. Let $m$ and $n$ be non-negative integers and let $r$ be any integer. Then
				\begin{equation*}\label{eq.glaqel1}
			\sum_{k = 0}^n \binom{m - n + k}{k} \binom{n - k}{r} (x + y)^{n-k-r} y^k = \sum_{k = 0}^n \binom{m + 1}{k}\binom{n - k}{r} x^{n - k - r} y^k.
			\end{equation*}
		\end{lemma}

\end{document}